\documentclass[preprint]{elsarticle}
\usepackage{amsmath,amsthm,amssymb}
\usepackage{times}
\usepackage{enumerate}
\usepackage{url}
\usepackage{lineno}
\usepackage{microtype}
\usepackage[american]{babel}
\usepackage{indentfirst}
\usepackage{geometry}
\usepackage{epsfig}
\newtheorem{result}{\textbf{Theorem}}
\newtheorem{proposition}{Proposition}[section]
\newtheorem{lemma}{Lemma}[section]
\newtheorem{definition}{Definition}[section]

\newtheorem{remark}{Remark}[section]
\newtheorem{example}{Example}[section]
\numberwithin{equation}{section}

\begin{document}
\baselineskip=15pt

\begin{frontmatter} 

\title{Homogenization of Viscous Sublinear Hamilton--Jacobi Equations with $u/\epsilon$-Dependence}

\author{Guyu Jin}
\address{Graduate School of Mathematical Sciences, The University of Tokyo, 3 Chome-8-1 Komaba, Meguro City, Tokyo 153-8902, Japan}
\ead{guyu567@g.ecc.u-tokyo.ac.jp} 


\begin{abstract}
We study the periodic homogenization of a class of viscous Hamilton--Jacobi equations with fast dependence on the unknown. This problem combines features of first-order Hamilton--Jacobi equations with \(u^\epsilon/\epsilon\)-periodic Hamiltonians and semilinear heat equations with rapidly oscillating positive potentials. In this paper, we prove qualitative homogenization results for $H(y,s,p)=F(s,p)+\eta W(y,s,p)$ when $|\eta|$ is sufficiently small, with the smallness threshold depending on the Lipschitz constant of the initial datum, and establish a large-time averaging result for a general class of evolutionary cell problems.
\end{abstract}

\end{frontmatter}

\section{Introduction}

\subsection{Setting of the problem}

Let $N\in\mathbb N$ and let $\epsilon>0$ be a small parameter. We consider a continuous Hamiltonian
\begin{equation*}
H:\mathbb R^N\times\mathbb R\times\mathbb R^N\to\mathbb R,
\end{equation*}
and an initial datum $u_0:\mathbb R^N\to\mathbb R.$
The unknown is a function $u^\epsilon:\mathbb R^N\times[0,\infty)\to\mathbb R,$
and we denote by $u_t^\epsilon$, $Du^\epsilon$ and $\Delta u^\epsilon$ its time derivative, spatial gradient and Laplacian, respectively. The equation studied in this paper is of the form
\begin{equation}
\begin{cases}
u_t^\epsilon
+
H\left(
\dfrac{x}{\epsilon},
\dfrac{u^\epsilon}{\epsilon},
Du^\epsilon
\right)
=
\epsilon\Delta u^\epsilon
&\text{in }\mathbb R^N\times(0,\infty),
\\[2mm]
u^\epsilon(x,0)=u_0(x)
&\text{on }\mathbb R^N,
\end{cases}\label{HJ}
\end{equation}
where
\[
u_0\in BUC(\mathbb R^N)\cap W^{1,\infty}(\mathbb R^N).
\]
Here $H$ is periodic in the first two variables and belongs to $C(\mathbb T^N\times \mathbb T\times\mathbb R^N)$. 

In order to understand the contribution of this work, we recall that the theory of homogenization for first-order Hamilton–Jacobi equation started with the famous unpublished work of Lions, Papanicolaou and Varadhan \cite{LionsPapanicolaouVaradhan1987} who completely solved the problem in the case of periodic and coercive Hamiltonians. Since the literature on homogenization theory for Hamilton--Jacobi equations is vast, we focus here on the works most relevant to the present paper. For the first-order Hamilton-Jacobi equation with $u^\epsilon/\epsilon$-periodic Hamiltonians, we first refer to Imbert and Monneau~\cite{imbert2008homogenization} and Barles ~\cite{barles2007some}. Achdou and Patrizi ~\cite{achdou2011homogenization} obtain a rate of convergence for contact-type Hamilton-Jacobi equations, and then Mitake, Ni and Tran \cite{mitake2025quantitative} establish the optimal convergence rate for convex superlinear contact-type Hamilton-Jacobi equation. Ni \cite{ni2025quantitative} obtain a quantitative homogenization result for first-order ODEs with a form related to our model, and Ni, Wang and Yan \cite{ni2023viscosity} are the first to use the method of contact Hamilton systems and implicit variational principle to study the contact-type Hamilton-Jacobi equation with $u$-periodicity. We also refer to \cite{wang2017implicit, wang2019aubry, wang2019variational} for related works particularly from the viewpoint of weak KAM theory. The viscous Hamilton--Jacobi equation has a different structure, since the associated cell problem contains an elliptic operator and the coercivity is not always necessary. Related results include the work of Barles and Souganidis~\cite{barles2001space}, and the recent quantitative homogenization result of Qian, Sprekeler, Tran and Yu~\cite{qian2024optimal}. In \cite{qian2024optimal}, the obtained convergence rate $O(\epsilon^{1/2})$ is optimal. For further results on second-order Hamilton--Jacobi homogenization, we refer to~\cite{LionsSouganidis2005,LionsSouganidis2005Viscous,Evans1992,tran2021hamilton}.

The distinguishing feature of \eqref{HJ} is the simultaneous fast dependence on the spatial variable $x/\epsilon$ and on the phase $u^\epsilon/\epsilon$. If the viscous term is removed, one recovers the contact-type first-order Hamilton--Jacobi equation. On the other hand, if the Hamiltonian has no gradient dependence and is a positive periodic function of $u^\epsilon/\epsilon$, the equation is closely related to the semilinear heat equation studied by Cesaroni, Dirr and Novaga~\cite{cesaroni2017homogenization}. In that setting the effective speed is described by a one-dimensional cell ODE, or equivalently by a travelling wave speed. Related periodic travelling waves in heterogeneous media were studied, for instance, by Chen and Namah~\cite{ChenNamah2002}. The cell problem considered here is closely related to the equations satisfied by travelling-wave solutions of reaction--diffusion equations , see \cite{berestycki2005speed,berestycki2007generalized, zhang2022homogenization}.  

The presence of the viscous term in \eqref{HJ} creates a new difficulty. For a general Hamiltonian $H(y,s,p)$, the heuristic explanations  lead to the cell problem
\begin{equation}
\begin{aligned}
c(p)(\chi_p)_z
&+
H\left(y,\chi_p,D_y\chi_p+p(\chi_p)_z\right)
\\
&=
\Delta_y\chi_p
+
2p\cdot D^2_{yz}\chi_p
+
|p|^2(\chi_p)_{zz},
\end{aligned}\label{cell2}
\end{equation}
where $c(p)$ is a constant depending on $p\in\mathbb R^N$. If a periodic solution of \eqref{cell2} is available, and $c(p)$ is continuous, then Evans' perturbed test function method~\cite{Evans1989} can be used to prove homogenization. However, solving \eqref{cell2} for a general Hamiltonian is highly nontrivial. It is worth noting that this equation is degenerate when $p=0$. So far, for general $H$, it remains unclear whether the cell problem  \eqref{cell2} admits a solution. 

In this paper we therefore focus on the class
\begin{equation}
H_\eta(y,s,p)=F(s,p)+\eta W(y,s,p),
\label{intro-second-H}
\end{equation}
where $F$ is the principal part and $\eta W$ is a spatially oscillatory perturbation. The smallness of $\eta$ is local in the gradient variable: if $K\subset\mathbb R^N$ is the compact set on which the correctors are needed, then the construction gives a number $\eta_K>0$. In the homogenization theorem, this compact set is chosen from the Lipschitz constant of $u_0$, see Section 3 for more details.

The main novelty of this work lies in the treatment of a second-order contact-type Hamilton--Jacobi homogenization problem. In the first-order theory with \(u^\epsilon/\epsilon\)-periodic Hamiltonians, the viscous term is absent, while in the semilinear heat equation setting the oscillatory forcing depends only on the fast phase \(u^\epsilon/\epsilon\) and has no Hamilton--Jacobi gradient dependence. By contrast, here the equation contains the viscous term \(\epsilon\Delta u^\epsilon\), and the oscillatory Hamiltonian may depend simultaneously on the fast space variable, the fast phase, and the gradient. This leads from the one-dimensional travelling-wave ODE of the semilinear model to a periodic-parabolic corrector equation, see Section 3 \eqref{periodic parabolic}. Parabolic equations with gradient-dependent nonlinearities have been studied in other contexts, for instance in~\cite{FernandezCaraGonzalezBurgosGuerreroPuel2006}; to the best of our knowledge, this is the first perturbative construction of periodic-parabolic correctors for viscous Hamilton--Jacobi equations with simultaneous $x/\epsilon$- and $u^\epsilon/\epsilon$-oscillations and a gradient-dependent spatial perturbation.

\subsection{Main results}

We first consider the unperturbed Hamiltonian
\begin{equation*}
H(y,s,p)=F(s,p),
\end{equation*}
where $F: \mathbb R \times \mathbb R^N \to \mathbb R$ is assumed to satisfy the following conditions:
\begin{itemize}
\item[(F1).]\textbf{(Positivity).}
There exists $c_0>0$ such that
\begin{equation}
F(s,p)\ge c_0>0
\end{equation}
for $s\in\mathbb R$, $p\in\mathbb R^N$.
\item[(F2).]\textbf{(Monotonicity along rays).}
For $\rho>0$, $s\in\mathbb R$, $p\in\mathbb R^N$,
\begin{equation}
\rho D_pF(s,\rho p)\cdot p-F(s,\rho p)<0.
\end{equation}
Equivalently,
$$
\rho\longmapsto \frac{F(s,\rho p)}{\rho}
$$
is strictly decreasing.
\item [(F3).]\textbf{(Limit at infinity).}\begin{equation}
F_\infty(p):=\lim_{\rho\to+\infty}\frac{F(s,\rho p)}{\rho}\label{recession}
\end{equation}
exists locally uniformly in $p$ and uniformly in $s$, and $F_\infty \in C(\mathbb R^N)$.
\item [(F4).]\textbf{(Regularity).}
\begin{equation}
F\in \operatorname{Lip}(\mathbb R^{N+1})\cap C^{3,\alpha}(\mathbb R^{N+1}),
\end{equation}
and there exists $L>0$ such that
\begin{equation}
|F(s,p)-F(s,q)|\le L|p-q|,
\end{equation}
and $F(s,p)$ is 1-periodic in $s$.
\end{itemize}

For each $p\in\mathbb R^N$, we look for a speed $c(p)\in\mathbb R$ and a monotone profile $\chi_p\in C^2(\mathbb R)$ satisfying
\begin{equation*}
\chi_p(z+1)=\chi_p(z)+1,
\qquad
\chi_p'(z)>0,
\end{equation*}
and
\begin{equation}
|p|^2\chi_p''(z)=c(p)\chi_p'(z)+F(\chi_p(z),p\chi_p'(z)).\label{cell1}
\end{equation}
The function $\chi_p$ is unique up to translations in $z$, and the speed $c(p)$ is uniquely determined. We define the effective Hamiltonian by
\begin{equation*}
\overline H(p):=-c(p).
\end{equation*}
We explain a heuristic derivation of (\ref{cell1}) in Section 2.

\begin{result}[Homogenization in the unperturbed case]
Assume that $F$ satisfies \textup{(F1)}--\textup{(F4)}. Let $u^\epsilon \in C(\mathbb R^N\times [0,\infty))$ solve
\begin{equation}
\begin{cases}
u_t^\epsilon+F\left(\dfrac{u^\epsilon}{\epsilon},Du^\epsilon\right)=\epsilon\Delta u^\epsilon
&\text{in }\mathbb R^N\times(0,\infty),
\\[2mm]
u^\epsilon(x,0)=u_0(x)
&\text{on }\mathbb R^N.
\end{cases}
\end{equation}
Then $u^\epsilon$ converges locally uniformly to the viscosity solution $u$ of
\begin{equation}
\begin{cases}
u_t+\overline H(Du)=0
&\text{in }\mathbb R^N\times(0,\infty),
\\
u(x,0)=u_0(x)
&\text{on }\mathbb R^N.
\end{cases}
\end{equation}
\end{result}

We next consider the perturbed Hamiltonian \eqref{intro-second-H}.For each fixed $p$, the cell problem is to find a corrector $\chi_p^\eta(y,z)=z+w_p^\eta(y,z)$ and a constant $c_\eta(p)$ satisfying the following parabolic equation:
\begin{equation*}
\begin{aligned}
c_\eta(p)(\chi_p^\eta)_z
&+
F\left(\chi_p^\eta,D_y\chi_p^\eta+p(\chi_p^\eta)_z\right)
\\
&+
\eta W\left(y,\chi_p^\eta,D_y\chi_p^\eta+p(\chi_p^\eta)_z\right)
\\
&=
\Delta_y\chi_p^\eta
+
2p\cdot D^2_{yz}\chi_p^\eta
+
|p|^2(\chi_p^\eta)_{zz}.
\end{aligned}
\end{equation*}
We construct these correctors by linearizing at the one-dimensional profile, using the Fredholm alternative for the corresponding parabolic operator, and then applying a fixed point argument. The associated effective Hamiltonian is
\begin{equation*}
\overline H_\eta(p):=-c_\eta(p).
\end{equation*}

\begin{result}[Homogenization in the perturbative case]
Let
\begin{equation*}
L_0:=\operatorname{Lip}(u_0),
\qquad
K_0:=\overline{B_{L_0}(0)}.
\end{equation*}
Assume that $F$ satisfies \textnormal{(F1)}--\textnormal{(F4)}, that
\[
W\in C_{\mathrm{loc}}^{3,\alpha}
(\mathbb T^N\times\mathbb T\times\mathbb R^N),
\]
and that there exists $L>0$ such that
\[
|W(x,r,p)-W(y,s,q)|
\leq
L\big((1+|p|+|q|)|x-y|+|r-s|+|p-q|\big)
\]
for all $x,y\in\mathbb R^N$, $r,s\in\mathbb R$, and $p,q\in\mathbb R^N$. Then there exists $\eta_{K_0}>0$ such that, if $|\eta|\le\eta_{K_0}$, the solution $u^\epsilon \in C(\mathbb R^N\times [0,\infty))$ of
\begin{equation}
\begin{cases}
u_t^\epsilon+
F\left(\dfrac{u^\epsilon}{\epsilon},Du^\epsilon\right)
+
\eta W\left(\dfrac{x}{\epsilon},\dfrac{u^\epsilon}{\epsilon},Du^\epsilon\right)
=\epsilon\Delta u^\epsilon
&\text{in }\mathbb R^N\times(0,\infty),
\\[2mm]
u^\epsilon(x,0)=u_0(x)
&\text{on }\mathbb R^N
\end{cases}
\end{equation}
converges locally uniformly to the viscosity solution $u$ of
\begin{equation}
\begin{cases}
u_t+\overline H_\eta(Du)=0
&\text{in }\mathbb R^N\times(0,\infty),
\\
u(x,0)=u_0(x)
&\text{on }\mathbb R^N.
\end{cases}
\end{equation}
\end{result}

Our argument requires $|\eta|$ to be sufficiently small in order to make the fixed-point map contractive and to preserve the strict monotonicity of $\chi_p^\eta$ in the $z$-variable; see Section~3. The assumptions on $W$ are used to establish the comparison principle; see Appendix~A.

Finally, we discuss the large-time average of the cell evolution associated with a general periodic Hamiltonian satisfying the assumptions \textup{(H1)}--\textup{(H2)} in Section 4. If $w=w(\zeta,\tau)$ solves
\begin{equation*}
\begin{cases}
w_\tau+H(\zeta,p\cdot\zeta+w,p+D_\zeta w)=\Delta_\zeta w, &\text{in }\mathbb R^N\times(0,\infty),
\\[1mm]
w(\zeta,0)=0,&\text{on }\mathbb R^N
\end{cases}
\end{equation*}
then there exists a constant $\lambda(p)$ such that
\begin{equation*}
\frac{w(\zeta,\tau)}{\tau}\to\lambda(p)
\qquad\text{as }\tau\to+\infty,
\end{equation*}
locally uniformly in $\zeta$. For the perturbative Hamiltonian $H_\eta(y,s,p)=F(s,p)+\eta W(y,s,p)$, this long-time constant agrees with the cell speed:
\begin{equation*}
\lambda(p)=c_\eta(p)=-\overline H_\eta(p).
\end{equation*}
Thus the speed obtained from the corrector and the averaged constant obtained from the cell evolution coincide.

The paper is organized as follows. In Section 2, we study the unperturbed Hamiltonian $F(s,p)$. We derive the cell ODE, construct the monotone corrector, prove the uniqueness and continuity of the cell speed, and establish homogenization in this case. In Section 3, we treat $F(s,p)+\eta W(y,s,p)$ and construct correctors using periodic-parabolic Fredholm theory and a fixed point argument. In Section 4, we prove the large-time averaged cell result and identify the averaged constant with the corrector speed in the perturbative class. Section 5 gives examples. The appendices contain the comparison principle, a uniform bound and initial-trace lemma, and the phase-shift contact lemma used in the perturbed test function argument.

\section{Homogenization for the unperturbed case}

In this section, we begin with Hamiltonians of the form
\begin{equation}
 H(y,s,p)=F(s,p),
\end{equation}
that is, $H$ is independent of $y$.

Throughout this section, we work under assumptions \textup{(F1)--(F4)}.

\subsection{Heuristic explanations of the cell ODE}

We start with a formal computation. Suppose that
\begin{equation*}
v_p^\epsilon(x,t):=\epsilon \chi\left(\frac{p\cdot x+ct}{\epsilon}\right)
\end{equation*}
is a formal solution of (\ref{HJ}). Here $c$ is a constant depending on $p$. Then,
$$
Dv_p^\epsilon
=
p\,\chi'\left(\frac{p\cdot x+ct}{\epsilon}\right),
\qquad
\epsilon\Delta v_p^\epsilon
=
|p|^2\chi''\left(\frac{p\cdot x+ct}{\epsilon}\right),
\qquad
\partial_t v_p^\epsilon
=
c\,\chi'\left(\frac{p\cdot x+ct}{\epsilon}\right).
$$
This leads formally to the cell equation
\begin{equation*}
c\chi'(z)+F(\chi(z),p\chi'(z))=|p|^2\chi''(z),
\end{equation*}
where
$$
z=\frac{p\cdot x+ct}{\epsilon}.
$$

\medskip
\begin{definition}
 We say that $\chi_p$ is the corrector associated with $p$ in the unperturbed case if
$$
\chi_p\in C^2(\mathbb R)
$$
and $c(p)$ satisfies the following conditions:

\[
\begin{cases}
\chi_p(z+1)=\chi_p(z)+1,\qquad \chi_p'(z)>0
\quad\text{for each }z\in\mathbb R,\\[2mm]
\chi_p \text{ and } c(p) \text{ solve }

|p|^2\chi_p''(z)
=
c(p)\chi_p'(z)+F(\chi_p(z),p\chi_p'(z)).

\end{cases}
\]
The constant $c(p)$ is called the cell speed associated with $p$.
\end{definition}

\subsection{Construction of the one-dimensional corrector}

We first consider the case $p\ne0$. Let
$$
v(s):=\chi_p'(z)
\quad\text{with }s=\chi_p(z),
\quad
c=c(p).
$$
Since $\chi_p$ is strictly increasing, $v$ is $1$-periodic, and
$$
\chi_p''(z)
=
(v(s))'
=
v_s\chi_p'(z)
=
v_s v.
$$
Thus we obtain the first-order periodic ODE
\begin{equation*}
|p|^2v(s)v'(s)=c\,v(s)+F(s,pv(s)).
\end{equation*}
Set $\mu:=-c$. Then
\begin{equation*}
v'(s)=\frac{-\mu v(s)+F(s,pv(s))}{|p|^2v(s)}.
\end{equation*}
We look for a positive $1$-periodic solution $v$ of this equation. The normalization requires
\begin{equation*}
\int_0^1\frac{1}{v_\mu(s)}\,ds=1,
\end{equation*}
because $\chi_p(z+1)=\chi_p(z)+1$.

Define
\begin{equation*}
\Theta_\mu(s,v):=\frac{-\mu v+F(s,pv)}{|p|^2v},
\qquad v>0.
\end{equation*}
By (F1), $\Theta_\mu(s,v)\to+\infty$ as $
v\to0^+.$
Since (\ref{recession}),
we choose $\mu>F_\infty(p)$. Then there exists $M\gg1$ such that
\begin{equation*}
F(s,pv)\le \frac{\mu+F_\infty(p)}{2}\,v
\quad\text{for }v\ge M.
\end{equation*}
Therefore $\Theta_\mu(s,M)<0$. Choose $m>0$ sufficiently small so that $\Theta_\mu(s,m)>0$.
Hence $[m,M]$ is invariant under the flow of
$$
v'=\Theta_\mu(s,v).
$$
Let
$
P_\mu(v_0):=v(1;v_0).
$
Then
$$
P_\mu:[m,M]\to[m,M]
$$
is continuous. By Brouwer's fixed point theorem, there exists $v_0\in[m,M]$ such that
$
P_\mu(v_0)=v_0.
$
Thus there exists a positive periodic solution $v_\mu$.

We now prove uniqueness. Observe that
\begin{equation*}
\frac{d}{dv}\left(\frac{F(s,pv)}{v}\right)
=
\frac{vD_pF(s,pv)\cdot p-F(s,pv)}{v^2}<0.
\end{equation*}
Therefore
$
\Theta_\mu(s,v)
$
is strictly decreasing in $v$.

If $v_1,v_2$ are two positive periodic solutions, then by the uniqueness of the first-order ODE, $v_1$ cannot touch $v_2$. Assume that $v_2(s)>v_1(s)$ for each $s\in \mathbb R$.
Then
$
(v_2-v_1)'
=
\Theta_\mu(s,v_2)-\Theta_\mu(s,v_1)<0.
$
Therefore
$
v_2(1)-v_1(1)<v_2(0)-v_1(0),
$
which contradicts the periodicity.

To impose the normalization $\int_0^1 v_\mu(s)^{-1}\,ds=1$, we introduce the map below.
Set
\begin{equation*}
T_p(\mu):=\int_0^1\frac{1}{v_\mu(s)}\,ds,
\qquad
\mu>F_\infty(p).
\end{equation*}

\medskip
\begin{proposition}
 $T_p$ is continuous and strictly increasing in $\mu$. Moreover,
\begin{equation}
\lim_{\mu\downarrow F_\infty(p)}T_p(\mu)=0,
\qquad
\lim_{\mu\to+\infty}T_p(\mu)=+\infty.
\end{equation}
\label{monotonicity}
\end{proposition}
\begin{proof}
 Continuity follows from the standard continuous dependence of ODE solutions on parameters. If $\mu_2>\mu_1$, then
$
\Theta_{\mu_2}(s,v)<\Theta_{\mu_1}(s,v).
$

Let
$
w:=v_{\mu_2}-v_{\mu_1}.
$
Assume that $\max_{\mathbb R} w\ge0$ and $w$ attains its maximum at $s_0$. Then
$
w'(s_0)=0.
$
Because
$
v_{\mu_2}(s_0)\ge v_{\mu_1}(s_0),
$
we have
$
\Theta_{\mu_2}(s_0,v_{\mu_2}(s_0))
\le
\Theta_{\mu_2}(s_0,v_{\mu_1}(s_0))
<
\Theta_{\mu_1}(s_0,v_{\mu_1}(s_0)).
$
Therefore
$$
v_{\mu_2}'(s_0)<v_{\mu_1}'(s_0),
$$
which contradicts
$$
w'(s_0)=v_{\mu_2}'(s_0)-v_{\mu_1}'(s_0)=0.
$$
Thus
$
v_{\mu_2}(s)<v_{\mu_1}(s)
\text{ for every }s\in\mathbb R.
$
Consequently,
$
T_p(\mu_2)>T_p(\mu_1).
$

For $\mu\to+\infty$, suppose that $v_\mu$ attains its maximum at $s_1$. Then
$
v_\mu'(s_1)=0, 
\mu v_\mu(s_1)=F(s_1,pv_\mu(s_1)).
$
By (F4), there exists $L>0$ such that
$
F(s_1,pv_\mu(s_1))
\le F(s_1,0)+L|p|v_\mu(s_1).
$
Thus
\begin{equation*}
(\mu-L|p|)v_\mu(s_1)\le F(s_1,0).
\end{equation*}
We deduce that
$
v_\mu\to0
$
uniformly, so
$
T_p(\mu)\to+\infty
$
as $\mu\to+\infty$.

For $\mu\downarrow F_\infty(p)$, fix $R>0$. By (F2) and (F3), for every $s\in\mathbb R$,
\begin{equation*}
\frac{F(s,pR)}{R}>F_\infty(p).
\end{equation*}
Pick $\mu>F_\infty(p)$, which is close enough to $F_\infty(p)$, such that
\begin{equation*}
-\mu R+F(s,pR)>0.
\end{equation*}
Then
$
\Theta_\mu(s,R)>0.
$
If $v_\mu(s)<R$ and attains its minimum at $s_2$, then
$
v_\mu'(s_2)=0,
$
but
$
\Theta_\mu(s_2,v_\mu(s_2))>0,
$
which is a contradiction. Thus
$
v_\mu(s)\ge R.
$
Therefore
$
T_p(\mu)\le\frac1R.
$
Letting $R\to+\infty$ gives
$$
T_p(\mu)\to0.
$$
\end{proof}

By Proposition \ref{monotonicity}, there exists a unique $\mu(p)>F_\infty(p)$ such that
\begin{equation*}
T_p(\mu(p))=1.
\end{equation*}
Define 
\begin{equation}
c(p):=-\mu(p).
\end{equation}

We now recover $\chi_p(z)$. Write
$
v_p=v_{\mu(p)}.
$
Define
\begin{equation*}
z(s):=\int_0^s\frac{1}{v_p(\sigma)}\,d\sigma.
\end{equation*}
Since
$
z(1)=1,
$
and $z(s)$ is strictly increasing, it has an inverse, which we denote by
$
\chi_p(z)=s.
$
Then
$
\chi_p(z+1)=\chi_p(z)+1.
$
Moreover,
$
\chi_p'(z)=v_p(\chi_p(z))>0.
$
Also,
\begin{align*}
\chi_p''(z)
&=
v_p'(\chi_p(z))v_p(\chi_p(z)) \\
&=
\frac1{|p|^2}\left(c(p)\chi_p'(z)+F(\chi_p(z),p\chi_p'(z))\right).
\end{align*}
Thus $\chi_p$ is the desired corrector.

If $p=0$, then
\begin{equation*}
c(0)\chi_0'+F(\chi_0,0)=0.
\end{equation*}
Pick
\begin{equation*}
\mu(0)=\left(\int_0^1\frac{1}{F(s,0)}\,ds\right)^{-1},
\qquad
\mu(0)=-c(0).
\end{equation*}
Define
\begin{equation*}
z(s)=\mu(0)\int_0^s\frac{1}{F(\sigma,0)}\,d\sigma.
\end{equation*}
Then $z(1)=1$. Let
$
s=\chi_0(z).
$
Then
$
\chi_0'(z)=\frac{F(\chi_0(z),0)}{\mu(0)},
$
and $\chi_0$ is the corrector.

We have therefore proved the following result.

\begin{result}[\textbf{Cell ODE and effective speed}]
Assume that $F=F(s,p)$ satisfies \textup{(F1)--(F4)}. Then, for every $p\in\mathbb R^N$, there exists a unique constant $c(p)\in\mathbb R$ such that the cell ODE
\begin{equation*}
|p|^2\chi''(z)=c(p)\chi'(z)+F(\chi(z),p\chi'(z))
\end{equation*}
admits a solution $\chi=\chi_p\in C^2(\mathbb R)$ satisfying
\begin{equation*}
\chi_p(z+1)=\chi_p(z)+1,\qquad \chi_p'(z)>0.
\end{equation*}
The corrector $\chi_p$ is unique up to translations in $z$.
\end{result}
We fix the translation by imposing the normalization $\chi_p(0)=0$.

We next prove the continuity of the cell speed. Continuity away from $p=0$ follows from the standard continuous dependence of ODE solutions on parameters, whereas the case $p=0$ requires a separate compactness argument.
\begin{proposition}
 $p\to c(p)$ is continuous on $\mathbb R^N$.\label{continuity}
\end{proposition}

\begin{proof}
 Suppose that
$
p_n\to p.
$
If $p\ne0$, using the continuity of $F_\infty$ and the continuous dependence of ODE solutions on parameters, we have
$
v_{p_n,\mu}\to v_{p,\mu}
$
uniformly for fixed
$
\mu>F_\infty(p).
$
Thus
$
T_{p_n}(\mu)\to T_p(\mu).
$
Since $T_p(\mu)$ is strictly increasing,
$
\mu(p_n)\to\mu(p).
$
Therefore
$$
c(p_n)\to c(p)
$$
as $n\to\infty$.

It remains to prove the continuity at $p=0$. Let
$
p_n\to0.
$
We may assume that $p_n\ne0$ for every $n$. Set
$$
\mu_n:=\mu(p_n),
\qquad
v_n:=v_{p_n},
\qquad
a_n(s):=\frac{1}{v_n(s)}.
$$
Then
\begin{equation*}
\int_0^1 a_n(s)\,ds=1.
\end{equation*}
We first establish uniform upper and lower bounds for $v_n$.

Let
$
A:=\max_{s\in[0,1]}F(s,0),
$ 
$
a:=\min_{s\in[0,1]}F(s,0)>0.
$
Since
$
\int_0^1\frac{1}{v_n(s)}\,ds=1,
$
we have
$
\min_s v_n(s)\le1\le\max_s v_n(s).
$
Let $s_n^M$ be a maximum point of $v_n$. Then $v_n'(s_n^M)=0$, and the equation for $v_n$ gives
\begin{equation*}
\mu_n v_n(s_n^M)=F(s_n^M,p_nv_n(s_n^M)).
\end{equation*}
Using the Lipschitz continuity of $F$ in the $p$-variable, we obtain
\begin{equation*}
F(s_n^M,p_nv_n(s_n^M))
\le
A+L|p_n|v_n(s_n^M).
\end{equation*}
Moreover, since $v_n(s_n^M)\ge1$ and $F\ge c_0$, we have
\begin{equation*}
\mu_n
=
\frac{F(s_n^M,p_nv_n(s_n^M))}{v_n(s_n^M)}
\ge
\frac{c_0}{v_n(s_n^M)}.
\end{equation*}
A simpler lower bound follows from a minimum point. Let $s_n^m$ be a minimum point of $v_n$. Then $v_n'(s_n^m)=0$ and
\begin{equation*}
\mu_n v_n(s_n^m)=F(s_n^m,p_nv_n(s_n^m)).
\end{equation*}
Since $v_n(s_n^m)\le1$ and $F\ge c_0$, we get $\mu_n\ge c_0.$

On the other hand, from the maximum point identity and the estimate above,
\begin{equation*}
\mu_n v_n(s_n^M)
\le
A+L|p_n|v_n(s_n^M).
\end{equation*}
Hence, for $n$ large enough,
\begin{equation*}
(\mu_n-L|p_n|)v_n(s_n^M)\le A.
\end{equation*}
Since $\mu_n\ge c_0$, choosing $n$ large so that $L|p_n|\le c_0/2$, we obtain
\begin{equation}
v_n(s)\le v_n(s_n^M)\le \frac{2A}{c_0}
\quad\text{for all }s. \label{upper bound}
\end{equation}
Similarly, using the minimum point identity,
\begin{equation*}
\mu_n v_n(s_n^m)
=
F(s_n^m,p_nv_n(s_n^m))
\ge
a-L|p_n|v_n(s_n^m).
\end{equation*}
Thus
\begin{equation*}
(\mu_n+L|p_n|)v_n(s_n^m)\ge a.
\end{equation*}
Since $\mu_n\le A+L|p_n|$, we get, for $n$ large,
\begin{equation}
v_n(s)\ge v_n(s_n^m)\ge C^{-1}
\quad\text{for all }s, \label{lower bound}
\end{equation}
where $C>0$ is independent of $n$. Therefore both $v_n$ and $a_n=1/v_n$ are uniformly bounded in $L^\infty(0,1)$.

By passing to a subsequence, we may assume that
$
\mu_n\to \mu_*
$
and
$$
a_n \rightharpoonup a_*
\quad\text{weakly-* in }L^\infty(0,1).
$$
We claim that
\begin{equation}
a_*(s)F(s,0)=\mu_*
\quad\text{a.e. in }(0,1).
\end{equation}

Indeed, dividing the ODE
$$
|p_n|^2 v_n(s)v_n'(s)
=
-\mu_n v_n(s)+F(s,p_nv_n(s))
$$
by $v_n(s)$ gives
\begin{equation*}
|p_n|^2v_n'(s)
=
-\mu_n+\frac{F(s,p_nv_n(s))}{v_n(s)}
=
-\mu_n+a_n(s)F(s,p_nv_n(s)).
\end{equation*}
Let $\varphi\in C^1_{\rm per}([0,1])$. Multiplying by $\varphi$ and integrating by parts, we get
\begin{align*}
&\int_0^1
\left(
-\mu_n+a_n(s)F(s,p_nv_n(s))
\right)
\varphi(s)\,ds
\\
&\qquad
=
|p_n|^2\int_0^1 v_n'(s)\varphi(s)\,ds
=
-|p_n|^2\int_0^1 v_n(s)\varphi'(s)\,ds.
\end{align*}
Since $v_n$ is uniformly bounded and $p_n\to0$, the right-hand side tends to $0$. Moreover,
\begin{equation*}
\left|
a_n(s)F(s,p_nv_n(s))-a_n(s)F(s,0)
\right|
\le
L|p_n|
\end{equation*}
because $a_n(s)v_n(s)=1$. Hence
\begin{equation*}
\int_0^1
\left(
-\mu_n+a_n(s)F(s,0)
\right)
\varphi(s)\,ds
\to0.
\end{equation*}
Letting $n\to\infty$, we obtain
\begin{equation*}
\int_0^1
\left(
-\mu_*+a_*(s)F(s,0)
\right)
\varphi(s)\,ds
=0
\end{equation*}
for every $\varphi\in C^1_{\rm per}([0,1])$. Therefore $a_*(s)F(s,0)=\mu_*$ a.e.

On the other hand, since
$
\int_0^1 a_n(s)\,ds=1,
$
we also have $\int_0^1 a_*(s)\,ds=1.$
Using
$
a_*(s)=\frac{\mu_*}{F(s,0)},
$
we get
\begin{equation*}
1
=
\mu_*\int_0^1\frac{1}{F(s,0)}\,ds.
\end{equation*}
Thus
\begin{equation*}
\mu_*
=
\left(
\int_0^1\frac{1}{F(s,0)}\,ds
\right)^{-1}
=
\mu(0).
\end{equation*}
Since every convergent subsequence of $\mu(p_n)$ has the same limit $\mu(0)$, we conclude that $\mu(p_n)\to\mu(0)
\text{ as }p_n\to0.$
Therefore
\begin{equation*}
c(p_n)=-\mu(p_n)\to-\mu(0)=c(0).
\end{equation*}
This proves the continuity of $c(p)$ at $p=0$.
\end{proof}
We end this subsection with uniform estimates for the correctors.
\begin{proposition}
 Suppose that $K$ is a compact subset of $\mathbb R^N$. Then there exist constants
$$
0<m_K\le M_K<\infty,
\qquad
C_K<\infty,
$$
such that
\begin{equation*}
m_K\le \chi_p'(z)\le M_K
\end{equation*}
for each $z\in\mathbb R$, and
\begin{equation*}
\|\chi_p-z\|_{C^2(\mathbb T)}\le C_K.
\end{equation*}\label{Bound}
\end{proposition}
\begin{proof}
We first consider the case where $K$ is contained in $\mathbb R^N\setminus\{0\}$. In this case, the equation
\begin{equation*}
v'(s)=\frac{-\mu v(s)+F(s,pv(s))}{|p|^2v(s)}
\end{equation*}
is a regular periodic ODE with parameters $(p,\mu)$. Hence the desired estimates follow from the continuous dependence of ODE solutions on parameters and the compactness of $K$.

It remains to prove the estimates near $p=0$. Let
\begin{equation*}
\mu_0:=\mu(0)
=
\left(
\int_0^1\frac{1}{F(s,0)}\,ds
\right)^{-1}.
\end{equation*}
By Proposition \ref{continuity},
$
\mu(p)\to\mu_0
\text{ as }p\to0.
$
Hence, there exist $\delta_0>0$ and constants $0<\mu_-<\mu_+$ such that $\mu_-\le\mu(p)\le\mu_+$ for $|p|\le\delta_0$.

Let $p\ne0$ and $|p|\le\delta_0$. Write
$
v_p:=v_{\mu(p)}.
$
We first prove that $v_p$ is uniformly bounded from above and below. By (\ref{upper bound}) and (\ref{lower bound}), there exist constants $0<m_0\le M_0<\infty$ such that
\begin{equation*}
m_0\le v_p(s)\le M_0
\end{equation*}
for all $s\in\mathbb T$ and all $0<|p|\le\delta_0$.

Next, we prove a uniform bound for $v_p'$. For fixed $p$ and $s$, define $V_p(s)>0$ by the algebraic equation
\begin{equation}
\mu(p)V_p(s)=F(s,pV_p(s)).\label{algebraic}
\end{equation}
This positive solution is unique. Indeed,
$
v\to \frac{F(s,pv)}{v}
$
is strictly decreasing by \textup{(F2)}, tends to $+\infty$ as $v\to0^+$, and tends to $F_\infty(p)$ as $v\to+\infty$. Since $\mu(p)>F_\infty(p)$, the equation has a unique positive solution.

We claim that $V_p$ is uniformly bounded from above and below for $|p|$ small. By (\ref{algebraic}) and the Lipschitz continuity of $F$,
\begin{equation*}
\mu(p)V_p(s)
\le
A+L|p|V_p(s).
\end{equation*}
Since $\mu(p)\ge c_0$ for $|p|$ small, we obtain
\begin{equation}
V_p(s)\le C.
\end{equation}
Similarly,
\begin{equation}
\mu(p)V_p(s)
=
F(s,pV_p(s))
\ge
a-L|p|V_p(s),
\end{equation}
and since $\mu(p)$ is bounded from above, we obtain
\begin{equation*}
V_p(s)\ge C^{-1}.
\end{equation*}
Thus $V_p$ is uniformly bounded above and below. In particular,
$
pV_p(s)\to0
$
uniformly as $p\to0$.

We also have a uniform bound for $V_p'$. Differentiating (\ref{algebraic}) with respect to $s$, we obtain
\begin{equation*}
\left(\mu(p)-D_pF(s,pV_p(s))\cdot p\right)V_p'(s)
=
F_s(s,pV_p(s)).
\end{equation*}
Since $\mu(p)\to\mu_0>0$ and $D_pF$ is bounded, after decreasing $\delta_0$ if necessary,
\begin{equation*}
\mu(p)-D_pF(s,pV_p(s))\cdot p\ge \gamma_0>0.
\end{equation*}
Therefore
\begin{equation*}
\|V_p'\|_{L^\infty(\mathbb T)}\le C.
\end{equation*}

Set
\begin{equation*}
G_p(s,v):=F(s,pv)-\mu(p)v.
\end{equation*}
Then $G_p(s,V_p(s))=0.$
Moreover, for $v$ in a fixed compact interval containing the ranges of $v_p$ and $V_p$, we have
\begin{equation*}
\partial_vG_p(s,v)
=
D_pF(s,pv)\cdot p-\mu(p)
\le
-\gamma_1<0
\end{equation*}
for all $0<|p|\le\delta_0$, after decreasing $\delta_0$ if necessary.
We now show that
\begin{equation*}
\|v_p-V_p\|_{L^\infty(\mathbb T)}\le C|p|^2.
\end{equation*}
Let
$
h_p:=v_p-V_p.
$
Assume that $h_p$ has a positive maximum at $s_0$. Then
$
h_p'(s_0)=0,
$
and hence
$
v_p'(s_0)=V_p'(s_0).
$
Since $v_p$ satisfies
\begin{equation}
|p|^2v_p(s)v_p'(s)=G_p(s,v_p(s)), \label{G_p eqn}
\end{equation}
we get
\begin{equation*}
G_p(s_0,v_p(s_0))
=
|p|^2v_p(s_0)V_p'(s_0).
\end{equation*}
On the other hand, since $v_p(s_0)>V_p(s_0)$ and $\partial_vG_p\le-\gamma_1$,
\begin{equation*}
G_p(s_0,v_p(s_0))
\le
-\gamma_1 h_p(s_0).
\end{equation*}
Therefore
\begin{equation*}
\gamma_1 h_p(s_0)
\le
|p|^2 |v_p(s_0)|\,|V_p'(s_0)|
\le C|p|^2.
\end{equation*}
Hence
\begin{equation*}
\max_{\mathbb T}h_p\le C|p|^2.
\end{equation*}

Similarly, assume that $h_p$ has a negative minimum at $s_1$. Then
$
h_p'(s_1)=0,
$
and hence
$
v_p'(s_1)=V_p'(s_1).
$
Since $v_p(s_1)<V_p(s_1)$ and $\partial_vG_p\le-\gamma_1$, we have
\begin{equation*}
G_p(s_1,v_p(s_1))
\ge
\gamma_1\bigl(V_p(s_1)-v_p(s_1)\bigr)
=
-\gamma_1 h_p(s_1).
\end{equation*}
On the other hand,
\begin{equation*}
G_p(s_1,v_p(s_1))
=
|p|^2v_p(s_1)V_p'(s_1).
\end{equation*}
Thus
\begin{equation*}
-\gamma_1 h_p(s_1)
\le
|p|^2|v_p(s_1)|\,|V_p'(s_1)|
\le C|p|^2.
\end{equation*}
Therefore
\begin{equation*}
-\min_{\mathbb T}h_p\le C|p|^2.
\end{equation*}
Combining the two estimates, we obtain
\begin{equation*}
\|v_p-V_p\|_{L^\infty(\mathbb T)}\le C|p|^2.
\end{equation*}
Finally, using (\ref{G_p eqn}) and the mean value theorem,
\begin{equation*}
|G_p(s,v_p(s))|
=
|G_p(s,v_p(s))-G_p(s,V_p(s))|
\le
C|v_p(s)-V_p(s)|
\le C|p|^2.
\end{equation*}
Since $v_p\ge m_0>0$, we conclude that
\begin{equation*}
|v_p'(s)|\le C
\end{equation*}
uniformly for all $s\in\mathbb T$ and all $0<|p|\le\delta_0$.

Now recall that $\chi_p'(z)=v_p(\chi_p(z)).$
Therefore $m_0\le \chi_p'(z)\le M_0.$
Moreover, $\chi_p''(z)
=
v_p'(\chi_p(z))\chi_p'(z).$
The uniform bounds on $v_p'$ and $v_p$ give
\begin{equation*}
|\chi_p''(z)|\le C.
\end{equation*}
Since we choose the normalization
$
\chi_p(0)=0,
$
and since
$
\chi_p(z+1)=\chi_p(z)+1,
$
the function $\chi_p(z)-z$ is periodic. The bounds on $\chi_p'$ and $\chi_p''$ imply
\begin{equation*}
\|\chi_p-z\|_{C^2(\mathbb T)}\le C.
\end{equation*}

Combining the estimates near $p=0$ with the standard compactness argument on
$
K\cap\{|p|\ge\delta_0\},
$
we obtain constants
$
0<m_K\le M_K<\infty,
$
$
C_K<\infty,
$
such that
\begin{equation*}
m_K\le \chi_p'(z)\le M_K
\end{equation*}
and
\begin{equation*}
\|\chi_p-z\|_{C^2(\mathbb T)}\le C_K
\end{equation*}
for all $p\in K$ and $z\in\mathbb R$.
\end{proof}

\subsection{Homogenization for the unperturbed case}
Before proving the homogenization theorem, we recall Lemma~\ref{phase-shift contact lemma}. This phase-shift lemma will also be used in the perturbative case, and its underlying idea goes back to \cite{imbert2008homogenization}. In the unperturbed case, the test function $\Phi^\epsilon_s$ is independent of the fast spatial variable $y$.

We now prove the homogenization theorem for $H(y,s,p)=F(s,p)$.

\begin{result}
 Assume that $F(s,p)$ satisfies (F1)-(F4). Consider
\begin{equation}
\begin{cases}
u_t^\epsilon+F\left(\dfrac{u^\epsilon}{\epsilon},Du^\epsilon\right)
=
\epsilon\Delta u^\epsilon
&\text{in }\mathbb R^N\times(0,\infty),\\[2mm]
u^\epsilon(x,0)=u_0(x)
&\text{on }\mathbb R^N.
\end{cases}
\end{equation}
Suppose that
$$
u_0\in BUC(\mathbb R^N)\cap W^{1,\infty}(\mathbb R^N),
$$
and $F$ satisfies (F1)--(F4). Then
$$
u^\epsilon\to u
$$
locally uniformly, where $u$ is the viscosity solution of
\begin{equation}
\begin{cases}
u_t-c(Du)=0
&\text{in }\mathbb R^N\times(0,\infty),\\
u(x,0)=u_0(x)
&\text{on }\mathbb R^N.
\end{cases}
\end{equation}
\end{result}

\medskip
\begin{proof}
 By Lemma~\ref{lem:initial-trace}, the family $\{u^\epsilon\}$ is locally uniformly bounded on finite time intervals, and the half-relaxed limits satisfy the initial condition.
 Define
$$
u^*:={\limsup}_{\epsilon\to0}^{*}u^\epsilon,
\qquad
u_*:={\liminf}_{\epsilon\to0,*}u^\epsilon.
$$
It suffices to show that $u^*$ is a viscosity subsolution. The proof that $u_*$ is the viscosity supersolution is similar.

Pick
$
\phi\in C^2((0,\infty)\times\mathbb R^N),
$
and suppose that $u^*-\phi$ attains its strict maximum at $(x_0,t_0)$, where $t_0>0$. Denote
$
a:=\phi_t(x_0,t_0),
$ $
p:=D\phi(x_0,t_0).
$
We only need to show
$
a\le c(p).
$

Pick the corrector $\chi_p$. Consider the perturbed test function
\begin{equation*}
\Phi_s^\epsilon(t,x)
=
\epsilon\chi_p\left(\frac{\phi(t,x)}{\epsilon}+s\right).
\end{equation*}
Then
$$
\Phi_{s+1}^\epsilon=\Phi_s^\epsilon+\epsilon.
$$
By Lemma \ref{phase-shift contact lemma}, applied to
$$
\chi(z)=\chi_p(z),
$$
there exist $s_\epsilon$ and maximum points $(x_\epsilon,t_\epsilon)\to(x_0,t_0)$ such that $M_\epsilon(s_\epsilon)=0,$
and $\frac{u^\epsilon(x_\epsilon,t_\epsilon)}{\epsilon}
=
\chi_p\left(
\frac{\phi(x_\epsilon,t_\epsilon)}{\epsilon}+s_\epsilon
\right).$
Let
$
z_\epsilon:=
\frac{\phi(x_\epsilon,t_\epsilon)}{\epsilon}+s_\epsilon,
$ $
z:=
\frac{\phi(x,t)}{\epsilon}+s_\epsilon.
$
Then
\begin{equation*}
\frac{u^\epsilon(x_\epsilon,t_\epsilon)}{\epsilon}
=
\frac{\Phi_{s_\epsilon}^\epsilon(x_\epsilon,t_\epsilon)}{\epsilon}
=
\chi_p(z_\epsilon).
\end{equation*}
Moreover,
$
(\Phi_{s_\epsilon}^\epsilon)_t
=
\chi_p'(z)\phi_t,
$ $
D\Phi_{s_\epsilon}^\epsilon
=
\chi_p'(z)D\phi,
$
and
$
\epsilon\Delta\Phi_{s_\epsilon}^\epsilon
=
\chi_p''(z)|D\phi|^2
+
\epsilon\chi_p'(z)\Delta\phi.
$

Suppose that $u^\epsilon-\Phi_{s_\epsilon}^\epsilon$ attains its maximum at $(x_\epsilon,t_\epsilon)$. The viscosity subsolution inequality gives
\begin{equation*}
0\ge
(\Phi_{s_\epsilon}^\epsilon)_t
+
F\left(\chi_p(z_\epsilon),D\Phi_{s_\epsilon}^\epsilon\right)
-
\epsilon\Delta\Phi_{s_\epsilon}^\epsilon.
\end{equation*}
By the cell ODE,
\begin{align*}
0\ge\;&
\chi_p'(z_\epsilon)\bigl(\phi_t-c(p)\bigr)
\\
&+
\left[
F\left(\chi_p(z_\epsilon),\chi_p'(z_\epsilon)D\phi\right)
-
F\left(\chi_p(z_\epsilon),\chi_p'(z_\epsilon)p\right)
\right]
\\
&-
\chi_p''(z_\epsilon)\left(|D\phi|^2-|p|^2\right)
-
\epsilon\chi_p'(z_\epsilon)\Delta\phi .
\end{align*}
Letting $\epsilon\to0$, we have
$
D\phi\to p,
$ $
\phi_t\to a.
$
Therefore
\begin{equation*}
\liminf_{\epsilon\to0}\chi_p'(z_\epsilon)(a-c(p))\le0.
\end{equation*}
Because
$
\inf_z \chi_p'(z)>0,
$
we obtain
$$
a\le c(p).
$$

Similarly, $u_*$ is the viscosity supersolution. With the comparison principle,
$$
u^*\le u_*.
$$
Since
$$
u_*\le u^*,
$$
we conclude that
$$
u:=u^*=u_*
$$
is the viscosity solution of the effective equation. This proves the local uniform convergence.
\end{proof}
The proof is based on Evans' perturbed test function method~\cite{Evans1989}.

\section{Homogenization for the perturbative case}

We now introduce a spatially oscillatory perturbation of the unperturbed case.

Consider
\begin{equation}
H(y,s,p)=F(s,p)+\eta W(y,s,p),
\end{equation}
where $\eta$ will be chosen sufficiently small on the relevant compact set of gradients. We assume that $F$ satisfies \textnormal{(F1)}--\textnormal{(F4)},
that
\[
W\in C_{\mathrm{loc}}^{3,\alpha}
(\mathbb T^N\times\mathbb T\times\mathbb R^N),
\]
and that there exists $L>0$ such that
\[
|W(x,r,p)-W(y,s,q)|
\le L\big((1+|p|+|q|)|x-y|+|r-s|+|p-q|\big)
\]
for all $x,y\in\mathbb R^N$, $r,s\in\mathbb R$, and $p,q\in\mathbb R^N$.

We write
$$
\chi_p^\eta:\mathbb R^N\times\mathbb R\to\mathbb R
$$
for the corrector associated with
$
H(y,s,p)=F(s,p)+\eta W(y,s,p).
$
The corresponding cell equation is
\begin{equation}
\begin{aligned}
c_\eta(p)(\chi_p^\eta)_z
&+
F\left(\chi_p^\eta,D_y\chi_p^\eta+p(\chi_p^\eta)_z\right)
\\
&+
\eta W\left(y,\chi_p^\eta,D_y\chi_p^\eta+p(\chi_p^\eta)_z\right)
\\
&=
\Delta_y\chi_p^\eta
+
2p\cdot D^2_{yz}\chi_p^\eta
+
|p|^2(\chi_p^\eta)_{zz}.
\end{aligned}\label{Cell}
\end{equation}
Set
$$
w_p^\eta(y,z):=\chi_p^\eta(y,z)-z.
$$
We seek a corrector of the form $\chi_p^\eta(y,z)=z+w_p^\eta(y,z)$ with $w_p^\eta$ periodic in $(y,z)$. This motivates the following definition.
\begin{definition}
We say that $\chi_p^\eta$ is a corrector associated with $p$ if
\[
\chi_p^\eta(y,z)=z+w_p^\eta(y,z),
\qquad
w_p^\eta\in C^2(\mathbb T^{N+1}),
\]
and there exists a constant $c_\eta(p)\in\mathbb R$ such that
\begin{equation}
\left\{
\begin{aligned}
&\chi_p^\eta(y,z+1)=\chi_p^\eta(y,z)+1,\qquad
(\chi_p^\eta)_z(y,z)>0,
\\[2mm]
&c_\eta(p)(\chi_p^\eta)_z
+
F\left(\chi_p^\eta,D_y\chi_p^\eta+p(\chi_p^\eta)_z\right)
\\
&\quad
+
\eta W\left(y,\chi_p^\eta,D_y\chi_p^\eta+p(\chi_p^\eta)_z\right)
\\
&\quad
=
\Delta_y\chi_p^\eta
+
2p\cdot D^2_{yz}\chi_p^\eta
+
|p|^2(\chi_p^\eta)_{zz}.
\end{aligned}
\right.
\end{equation}
The constant $c_\eta(p)$ is called the cell speed associated with $p$.
\end{definition}
Let $\alpha_0\in(0,1)$ denote the Hölder exponent in the regularity assumptions on $F$ and $W$. Throughout this section, we fix an exponent $\alpha\in(0,\alpha_0)$.
\subsection{Construction of a corrector}

Let
$$
w_p^0(z):=\chi_p^0(z)-z,
\qquad
w(y,z):=w_p^0(z)+\psi(y,z),
\qquad
c:=c_0(p)+\kappa.
$$
We define the nonlinear operator
$
\mathcal N:\mathbb R^N\times C^{2,\alpha}(\mathbb T^{N+1})\times\mathbb R\times\mathbb R
\to C^\alpha(\mathbb T^{N+1})
$
by
\begin{equation*}
\begin{aligned}
\mathcal N(p,\psi,\kappa,\eta)
:=&\;
(c_0+\kappa)(1+w_z)
+
F\left(z+w,D_yw+p(1+w_z)\right)
\\
&+
\eta W\left(y,z+w,D_yw+p(1+w_z)\right)
\\
&-
\Delta_yw
-
2p\cdot D^2_{yz}w
-
|p|^2w_{zz}.
\end{aligned}
\end{equation*}
We require
\begin{equation*}
\mathcal N(p,\psi,\kappa,\eta)=0,
\qquad
\int_{\mathbb T^{N+1}}\psi=0.
\end{equation*}
Clearly,
$
\mathcal N(p,0,0,0)=0.
$
Denote
\begin{equation*}
a_p(z):=(\chi_p^0)'(z)=1+(w_p^0)'(z)>0.
\end{equation*}
We linearize $\mathcal N$ at $(\psi,\kappa,\eta)=(0,0,0)$ in the variables $(\psi,\kappa)$.

For $\theta\in\mathbb R$, we have
\begin{align*}
\mathcal N(p,\theta\psi,\theta\kappa,0)
=&\;
c_0a_p
+
\theta(c_0\psi_z+\kappa a_p)
+
O(\theta^2)
+
F(\chi_p^0,pa_p)
\\
&+
\theta\left[
F_s(\chi_p^0,pa_p)\psi
+
D_pF(\chi_p^0,pa_p)\cdot(D_y\psi+p\psi_z)
\right]
+
O(\theta^2)
\\
&-
\theta\left(
\Delta_y\psi
+
2p\cdot D^2_{yz}\psi
+
|p|^2\psi_{zz}
\right)
-
|p|^2(w_p^0)''.
\end{align*}
Thus,
\begin{equation*}
\begin{aligned}
\left.\frac{d}{d\theta}\mathcal N(p,\theta\psi,\theta\kappa,0)\right|_{\theta=0}
=&\;
c_0(p)\psi_z
+
F_s(\chi_p^0,pa_p)\psi
\\
&+
D_pF(\chi_p^0,pa_p)\cdot(D_y\psi+p\psi_z)
\\
&-
\Delta_y\psi
-
2p\cdot D^2_{yz}\psi
-
|p|^2\psi_{zz}
+
\kappa a_p
\\
=:&\;
L_p\psi+\kappa a_p,
\end{aligned}
\end{equation*}
where $L_p$ is a linear operator.
Whenever necessary, we decrease the fixed Hölder exponent $\alpha$ without changing the notation.

For a function $\psi=\psi(y,z)$, we introduce the change of variables
$$
r=z-p\cdot y,
\qquad
\widetilde\psi(y,r):=\psi(y,p\cdot y+r).
$$
Then
$
\psi_z=\widetilde\psi_r,
$ $
D_y\psi+p\psi_z=D_y\widetilde\psi,
$
and
\begin{equation*}
\Delta_y\psi
+
2p\cdot D^2_{yz}\psi
+
|p|^2\psi_{zz}
=
\Delta_y\widetilde\psi.
\end{equation*}
Therefore, in the variables $(y,r)$, the linearized operator $L_p$ becomes the periodic-parabolic operator
\begin{equation}
\widetilde L_p\widetilde\psi
=
c_0(p)\widetilde\psi_r
+
F_s(\chi_p^0,p a_p)\widetilde\psi
+
D_pF(\chi_p^0,p a_p)\cdot D_y\widetilde\psi
-
\Delta_y\widetilde\psi, \label{periodic parabolic}
\end{equation}
where
$
a_p=(\chi_p^0)'.
$
Here $\chi_p^0$ and $a_p$ are evaluated at $p\cdot y+r$.

For fixed $p$, we define the skew torus $\mathbb T_p^{N+1}$ by the identifications
$$
(y,r)\sim(y+k,r-p\cdot k),
\qquad
(y,r)\sim(y,r+\ell),
\qquad
k\in\mathbb Z^N,\ \ell\in\mathbb Z.
$$
Equivalently, this is the usual periodicity of $\psi(y,z)$ in the variables $(y,z)$.

We say that
$$
\widetilde\psi\in C_{\rm per}^{2+\alpha,1+\alpha/2}(\mathbb T_p^{N+1})
$$
if $\widetilde\psi$ satisfies the above skew-periodicity and
$$
\widetilde\psi,\quad D_y\widetilde\psi,\quad D_y^2\widetilde\psi,\quad \widetilde\psi_r
$$
are continuous and skew-periodic, with $D_y^2\widetilde\psi$ and $\widetilde\psi_r$ Hölder continuous with respect to the parabolic distance
$$
d_{\rm par}\bigl((y,r),(y',r')\bigr)
=
|y-y'|+|r-r'|^{1/2}.
$$
The corresponding norm is
\begin{equation*}
\begin{aligned}
\|\widetilde\psi\|_{C^{2+\alpha,1+\alpha/2}}
:=
&\;
\|\widetilde\psi\|_{L^\infty}
+
\|D_y\widetilde\psi\|_{L^\infty}
+
\|D_y^2\widetilde\psi\|_{L^\infty}
+
\|\widetilde\psi_r\|_{L^\infty}
\\
&+
[D_y^2\widetilde\psi]_{\alpha,\alpha/2}
+
[\widetilde\psi_r]_{\alpha,\alpha/2}.
\end{aligned}
\end{equation*}
Here
\begin{equation*}
[f]_{\alpha,\alpha/2}
:=
\sup_{(y,r)\ne(y',r')}
\frac{|f(y,r)-f(y',r')|}
{|y-y'|^\alpha+|r-r'|^{\alpha/2}}.
\end{equation*}
We also define
\begin{equation*}
\mathcal X_p
:=
\left\{
\widetilde\psi\in
C_{\rm per}^{2+\alpha,1+\alpha/2}(\mathbb T_p^{N+1}):
\int_{\mathbb T_p^{N+1}}\widetilde\psi=0
\right\}
\times\mathbb R
\end{equation*}
and
\begin{equation*}
\mathcal Y_p
:=
C_{\rm per}^{\alpha,\alpha/2}(\mathbb T_p^{N+1}).
\end{equation*}
We define $\widetilde A_p:\mathcal X_p\to\mathcal Y_p$ 
by
\begin{equation*}
\widetilde A_p(\widetilde\psi,\kappa)
=
\widetilde L_p\widetilde\psi+\kappa a_p.
\end{equation*}
Before proving invertibility, we first record a lemma for parabolic equations on skew tori. For the periodic-parabolic Schauder estimates and Fredholm alternative used in the construction of the correctors, we refer to Lieberman~\cite{Lieberman1996}, Ladyzhenskaya--Solonnikov--Ural'tseva~\cite{LadyzhenskayaSolonnikovUraltseva1968}, and Hess~\cite{Hess1991}.

\begin{lemma}[\textbf{Uniform periodic-parabolic estimates on skew tori}]
Let $K\Subset\mathbb R^N$. For each $p\in K$, let $\mathbb T_p^{N+1}$ be the skew torus defined by the identifications
$$
(y,r)\sim(y+k,r-p\cdot k),
\qquad
(y,r)\sim(y,r+\ell),
\qquad
k\in\mathbb Z^N,\ \ell\in\mathbb Z.
$$
Consider the linear periodic-parabolic operator
\begin{equation*}
\mathcal P_p u
=
c_p u_r
+
b_p(y,r)\cdot D_yu
+
d_p(y,r)u
-
\Delta_yu
\end{equation*}
on $\mathbb T_p^{N+1}$. Assume that $|c_p|\ge c_K>0$
for all $p\in K$, and that
\begin{equation*}
\|b_p\|_{C^{\alpha,\alpha/2}(\mathbb T_p^{N+1})}
+
\|d_p\|_{C^{\alpha,\alpha/2}(\mathbb T_p^{N+1})}
\le C_K
\end{equation*}
uniformly for $p\in K$. Then there exists a constant $C_K>0$, independent of $p\in K$, such that every skew-periodic solution $u$ of
\begin{equation*}
\mathcal P_p u=f
\end{equation*}
satisfies
\begin{equation*}
\|u\|_{C^{2+\alpha,1+\alpha/2}(\mathbb T_p^{N+1})}
\le
C_K
\left(
\|f\|_{C^{\alpha,\alpha/2}(\mathbb T_p^{N+1})}
+
\|u\|_{L^\infty(\mathbb T_p^{N+1})}
\right).
\end{equation*}
Moreover, the periodic problem for $\mathcal P_p$ satisfies the Fredholm alternative with index $0$.\label{lem:uniform-skew-schauder}
\end{lemma}

\begin{proof}
We regard functions on $\mathbb T_p^{N+1}$ as functions on $\mathbb R^N\times\mathbb R$ satisfying the skew-periodicity
\begin{equation*}
u(y+k,r-p\cdot k)=u(y,r),
\qquad
u(y,r+1)=u(y,r).
\end{equation*}
The local parabolic Schauder estimate for
$$
c_pu_r+b_p\cdot D_yu+d_pu-\Delta_yu=f
$$
is applied on ordinary parabolic cylinders in the variables $(y,r)$. Its constant depends only on the lower bound of $|c_p|$, the Hölder norms of $b_p,d_p$, the Hölder exponent, and the dimension. By the assumptions, these quantities are uniformly bounded for $p\in K$.

It remains to explain why the local estimates can be patched together with constants independent of $p$. The skew-periodicity is generated by the translations
$$
(y,r)\mapsto(y+e_i,r-p_i),
\qquad i=1,\ldots,N,
$$
and
$$
(y,r)\mapsto(y,r+1).
$$
These transition maps are translations in the variables $(y,r)$; in particular, they do not mix the parabolic time variable $r$ with the spatial derivatives in the local equation. Hence the local parabolic Schauder estimates are invariant under these transitions.

Since $p$ ranges in the compact set $K$, the fundamental parallelepipeds associated with the lattice generated by
$
(e_i,-p_i),$ $ i=1,\ldots,N,
$
and
$
(0,1)
$
have uniformly bounded diameter and fixed volume. Therefore one can cover each fundamental domain by a number of parabolic cylinders bounded only in terms of $K$. Using the skew-periodicity, the local estimates on this finite covering give the global estimate on $\mathbb T_p^{N+1}$ with a constant independent of $p\in K$.

The Fredholm property follows from the usual periodic-parabolic theory. Indeed, the solution operator over one period in the time variable $r$ is compact on Hölder spaces by parabolic smoothing. Therefore the periodic solvability problem is equivalent to an equation of the form $I-P_p$ with $P_p$ compact. Hence $\mathcal P_p$ is Fredholm of index $0$, and the standard Fredholm alternative holds.
\end{proof}
\begin{proposition}
For every fixed $p\in\mathbb R^N$, the operator
$$
\widetilde A_p:
(\widetilde\psi,\kappa)\longmapsto
\widetilde L_p\widetilde\psi+\kappa a_p
$$
is an isomorphism from $\mathcal X_p$ to $\mathcal Y_p$.\label{isomorphism}
\end{proposition}

\begin{proof}
First, differentiating the unperturbed cell ODE
\begin{equation*}
|p|^2(\chi_p^0)''
=
c_0(p)(\chi_p^0)'
+
F(\chi_p^0,p(\chi_p^0)')
\end{equation*}
with respect to $z$, we obtain $L_pa_p=0.$
Equivalently, $\widetilde L_pa_p=0.$
We now prove that
\begin{equation*}
\ker \widetilde L_p=\operatorname{span}\{a_p\}.
\end{equation*}
Let $\widetilde L_p\widetilde\psi=0$ and write
$
\widetilde\psi=a_ph.
$
Using $\widetilde L_pa_p=0$, a direct computation gives
\begin{equation*}
c_0(p)h_r
+
\widetilde b_p(y,r)\cdot D_yh
-
\Delta_yh=0
\end{equation*}
on $\mathbb T_p^{N+1}$, where $\widetilde b_p$ is a bounded Hölder continuous vector field. Since $c_0(p)\ne0$, this is a linear periodic -parabolic equation with no zero-order term. By the strong maximum principle, every periodic solution is constant. Hence
$
\widetilde\psi=Ca_p.
$
Therefore
\begin{equation*}
\ker \widetilde L_p=\operatorname{span}\{a_p\}.
\end{equation*}

By Lemma~\ref{lem:uniform-skew-schauder}, applied with
$
\mathcal P_p=\widetilde L_p,
$
the operator
$$
\widetilde L_p:
C_{\rm per}^{2+\alpha,1+\alpha/2}(\mathbb T_p^{N+1})
\to
C_{\rm per}^{\alpha,\alpha/2}(\mathbb T_p^{N+1})
$$
is Fredholm of index $0$. Let $\rho_p^*>0$ be the positive normalized generator of the adjoint kernel:
\begin{equation*}
\widetilde L_p^*\rho_p^*=0,
\qquad
\int_{\mathbb T_p^{N+1}}\rho_p^*=1.
\end{equation*}
Then
$
\widetilde L_p\widetilde\psi=f
$
is solvable if and only if $\int_{\mathbb T_p^{N+1}}f\rho_p^*=0.$
Therefore
$$
\widetilde L_p\widetilde\psi+\kappa a_p=f
$$
is solvable if and only if
\begin{equation*}
\int_{\mathbb T_p^{N+1}}(f-\kappa a_p)\rho_p^*=0.
\end{equation*}
Thus the unique admissible value of $\kappa$ is
\begin{equation*}
\kappa
=
\frac{
\int_{\mathbb T_p^{N+1}}f\rho_p^*
}{
\int_{\mathbb T_p^{N+1}}a_p\rho_p^*
}.
\end{equation*}
The denominator is strictly positive because $a_p>0$ and $\rho_p^*>0$.

For this value of $\kappa$, the equation for $\widetilde\psi$ is solvable. The solution is unique after imposing the normalization
\begin{equation*}
\int_{\mathbb T_p^{N+1}}\widetilde\psi=0.
\end{equation*}
Indeed, the kernel is exactly $\operatorname{span}\{a_p\}$ and
$
\int_{\mathbb T_p^{N+1}}a_p>0.
$
Hence $\widetilde A_p$ is an isomorphism.
\end{proof}
We next prove the required uniform invertibility.
\begin{proposition}[\textbf{Uniform invertibility on compact sets}]
Let $K\subset\mathbb R^N$ be a compact set. There exists $B_K>0$ such that, for every $p\in K$ and every $f\in\mathcal Y_p$, the equation
\begin{equation*}
\widetilde L_p\widetilde\psi+\kappa a_p=f,
\qquad
\int_{\mathbb T_p^{N+1}}\widetilde\psi=0,
\end{equation*}
has a unique solution $(\widetilde\psi,\kappa)\in\mathcal X_p$, and
\begin{equation*}
\|\widetilde\psi\|_{C^{2+\alpha,1+\alpha/2}(\mathbb T_p^{N+1})}
+
|\kappa|
\le
B_K\|f\|_{C^{\alpha,\alpha/2}(\mathbb T_p^{N+1})}.
\end{equation*}
\end{proposition}

\begin{proof}
For each fixed $p$, the existence and uniqueness follow from Proposition \ref{isomorphism}. We prove that the inverse estimates are uniform for $p\in K$.

Suppose by contradiction that the estimate is false. Then there exist
$
p_n\in K,$ $ f_n\in\mathcal Y_{p_n},
$
and solutions
$
(\widetilde\psi_n,\kappa_n)\in\mathcal X_{p_n}
$
of
\begin{equation*}
\widetilde L_{p_n}\widetilde\psi_n+\kappa_n a_{p_n}=f_n,
\qquad
\int_{\mathbb T_{p_n}^{N+1}}\widetilde\psi_n=0
\end{equation*}
such that
\begin{equation*}
\|\widetilde\psi_n\|_{C^{2+\alpha,1+\alpha/2}(\mathbb T_{p_n}^{N+1})}
+
|\kappa_n|
=1,
\end{equation*}
while
\begin{equation*}
\|f_n\|_{C^{\alpha,\alpha/2}(\mathbb T_{p_n}^{N+1})}\to0.
\end{equation*}
By taking a subsequence, we may assume that
$
p_n\to p_\infty\in K
$
and
$
\kappa_n\to\kappa_\infty,
$ as $n \to \infty$.

We view each $\widetilde\psi_n$ as a function on $\mathbb R^{N+1}$ satisfying the skew-periodicity
\begin{equation*}
\widetilde\psi_n(y+k,r-p_n\cdot k)=\widetilde\psi_n(y,r),
\qquad
\widetilde\psi_n(y,r+1)=\widetilde\psi_n(y,r).
\end{equation*}
By the normalization
\begin{equation*}
\|\widetilde\psi_n\|_{C^{2+\alpha,1+\alpha/2}(\mathbb T_{p_n}^{N+1})}
+
|\kappa_n|
=1,
\end{equation*}
the functions $\widetilde\psi_n$ are uniformly bounded in
$
C_{\rm loc}^{2+\alpha,1+\alpha/2}(\mathbb R^N\times\mathbb R).
$
Therefore, by the compact embedding of parabolic Hölder spaces, after passing to a further subsequence, there exists a function $\widetilde\psi_\infty$ such that
\begin{equation*}
\widetilde\psi_n\to\widetilde\psi_\infty
\end{equation*}
locally in $C^{2+\beta,1+\beta/2}$ for every $0<\beta<\alpha$  as $n \to \infty$.

Since
$
p_n\to p_\infty,
$
the skew-periodicity passes to the limit. Thus, for every $k\in\mathbb Z^N$,
\begin{equation*}
\widetilde\psi_\infty(y+k,r-p_\infty\cdot k)
=
\widetilde\psi_\infty(y,r),
\qquad
\widetilde\psi_\infty(y,r+1)
=
\widetilde\psi_\infty(y,r).
\end{equation*}
Hence $\widetilde\psi_\infty$ descends to a function on $\mathbb T_{p_\infty}^{N+1}$.

The normalization also passes to the limit. Indeed, using the parametrization
$$
(y,r)=(\theta,\sigma-p_n\cdot\theta),
\qquad
(\theta,\sigma)\in[0,1]^N\times[0,1],
$$
of a fundamental domain of $\mathbb T_{p_n}^{N+1}$, we have
\begin{equation*}
0
=
\int_{\mathbb T_{p_n}^{N+1}}\widetilde\psi_n
=
\int_{[0,1]^{N+1}}
\widetilde\psi_n(\theta,\sigma-p_n\cdot\theta)\,d\theta d\sigma.
\end{equation*}
Letting $n\to\infty$, we get
\begin{equation*}
\int_{\mathbb T_{p_\infty}^{N+1}}\widetilde\psi_\infty=0.
\end{equation*}

We now pass to the limit in the equation. Since $p_n\to p_\infty$ and the unperturbed correctors depend continuously on $p$, the coefficients of $\widetilde L_{p_n}$ converge locally uniformly to those of $\widetilde L_{p_\infty}$. Since
$
f_n\to0
$
and
$$
\widetilde\psi_n\to\widetilde\psi_\infty
\quad\text{locally in }C^{2+\beta,1+\beta/2},
$$
we obtain
\begin{equation*}
\widetilde L_{p_\infty}\widetilde\psi_\infty
+
\kappa_\infty a_{p_\infty}
=0,
\qquad
\int_{\mathbb T_{p_\infty}^{N+1}}\widetilde\psi_\infty=0.
\end{equation*}
By Proposition \ref{lem:uniform-skew-schauder} applied to $p_\infty$, we conclude that $\widetilde\psi_\infty=0,
$ $
\kappa_\infty=0.$

We next show that this convergence is strong in the full parabolic Hölder norm. From the local uniform convergence and the skew-periodicity, we have
\begin{equation*}
\|\widetilde\psi_n\|_{L^\infty(\mathbb T_{p_n}^{N+1})}\to0.
\end{equation*}
Indeed, using the same parametrization of the fundamental domains,
$
(y,r)=(\theta,\sigma-p_n\cdot\theta),
$
this follows from the uniform convergence of
$
\widetilde\psi_n(\theta,\sigma-p_n\cdot\theta)
$
to
$
\widetilde\psi_\infty(\theta,\sigma-p_\infty\cdot\theta)=0
$
on the compact set $[0,1]^{N+1}$.

Applying Lemma~\ref{lem:uniform-skew-schauder} to
\begin{equation*}
\widetilde L_{p_n}\widetilde\psi_n
=
f_n-\kappa_n a_{p_n},
\end{equation*}
we obtain
\begin{equation}
\begin{aligned}
\|\widetilde\psi_n\|_{C^{2+\alpha,1+\alpha/2}(\mathbb T_{p_n}^{N+1})}
&\le
C_K
\left(
\|f_n-\kappa_n a_{p_n}\|_{C^{\alpha,\alpha/2}(\mathbb T_{p_n}^{N+1})}
+
\|\widetilde\psi_n\|_{L^\infty(\mathbb T_{p_n}^{N+1})}
\right)
\\
&\le
C_K
\left(
\|f_n\|_{C^{\alpha,\alpha/2}(\mathbb T_{p_n}^{N+1})}
+
|\kappa_n|
+
\|\widetilde\psi_n\|_{L^\infty(\mathbb T_{p_n}^{N+1})}
\right).
\end{aligned}
\end{equation}
The right-hand side tends to $0$. Hence
\begin{equation*}
\|\widetilde\psi_n\|_{C^{2+\alpha,1+\alpha/2}(\mathbb T_{p_n}^{N+1})}
+
|\kappa_n|
\to0,
\end{equation*}
which contradicts
\begin{equation*}
\|\widetilde\psi_n\|_{C^{2+\alpha,1+\alpha/2}(\mathbb T_{p_n}^{N+1})}
+
|\kappa_n|
=1.
\end{equation*}
Therefore, the desired uniform estimate holds.
\end{proof}

We now solve the nonlinear cell problem by a fixed point argument. For
$
\widetilde\psi\in\mathcal X_p,
$
we write
$
\psi(y,z):=\widetilde\psi(y,z-p\cdot y).
$
Equivalently,
$
\widetilde\psi(y,r)=\psi(y,p\cdot y+r).
$
We define
$$
\widetilde{\mathcal N}(p,\widetilde\psi,\kappa,\eta)
$$
to be the expression of
$
\mathcal N(p,\psi,\kappa,\eta)
$
in the variables $(y,r)$. Then
\begin{equation*}
\widetilde{\mathcal N}(p,0,0,0)=0,
\end{equation*}
and
\begin{equation*}
D_{(\widetilde\psi,\kappa)}
\widetilde{\mathcal N}(p,0,0,0)(\widetilde\psi,\kappa)
=
\widetilde A_p(\widetilde\psi,\kappa).
\end{equation*}
Set
$
Z:=(\widetilde\psi,\kappa).
$
Taylor's formula gives
\begin{equation*}
\widetilde{\mathcal N}(p,Z,\eta)
=
\widetilde A_pZ
+
\eta b_p
+
\widetilde R_p(Z,\eta),
\end{equation*}
where $b_p(y,r)
=
W\left(
y,
\chi_p^0(p\cdot y+r),
p\,a_p(p\cdot y+r)
\right).$
Here
$
a_p=(\chi_p^0)'.
$
For every compact set $K\Subset\mathbb R^N$, define
\begin{equation*}
M_K:=
\sup_{p\in K}\|b_p\|_{\mathcal Y_p}.
\end{equation*}
Then $M_K<\infty$.

Moreover, since $F$ and $W$ are $C^{3,\alpha}$ on the relevant compact set, there exist constants
$
r_K>0,
$ $
C_K^{(R)}>0
$
such that, whenever
$$
\|Z_i\|_{\mathcal X_p}\le r_K,
\qquad
|\eta|\le1,
\qquad
i=1,2,
$$
we have
\begin{equation*}
\|\widetilde R_p(Z_1,\eta)-\widetilde R_p(Z_2,\eta)\|_{\mathcal Y_p}
\le
C_K^{(R)}
\left(
\|Z_1\|_{\mathcal X_p}
+
\|Z_2\|_{\mathcal X_p}
+
|\eta|
\right)
\|Z_1-Z_2\|_{\mathcal X_p}.
\end{equation*}
In particular,
\begin{equation*}
\|\widetilde R_p(Z,\eta)\|_{\mathcal Y_p}
\le
C_K^{(R)}
\left(
\|Z\|_{\mathcal X_p}^2
+
|\eta|\|Z\|_{\mathcal X_p}
\right).
\end{equation*}
The fixed point argument yields the following local solution branch for the nonlinear cell problem.

\begin{proposition}
For every compact set $K\Subset\mathbb R^N$, there exist constants
$
\eta_K>0,
$ $
C_K>0
$
such that, for all $p\in K$ and $|\eta|\le\eta_K$, there exists a unique small solution
$$
Z_p^\eta=(\widetilde\psi_p^\eta,\kappa_p^\eta)\in\mathcal X_p
$$
of
\begin{equation*}
\widetilde{\mathcal N}(p,Z_p^\eta,\eta)=0.
\end{equation*}
Moreover,
\begin{equation*}
\|\widetilde\psi_p^\eta\|_{C^{2+\alpha,1+\alpha/2}}
+
|\kappa_p^\eta|
\le
C_K|\eta|.
\end{equation*}\label{estimate}
\end{proposition}

\begin{proof}
The equation
$
\widetilde{\mathcal N}(p,Z,\eta)=0
$
is equivalent to $\widetilde A_pZ
=
-\eta b_p-\widetilde R_p(Z,\eta).$
Therefore,
\begin{equation*}
Z
=
-\widetilde A_p^{-1}
\left[
\eta b_p+\widetilde R_p(Z,\eta)
\right].
\end{equation*}
Define
\begin{equation*}
\mathcal T_p(Z)
:=
-\widetilde A_p^{-1}
\left[
\eta b_p+\widetilde R_p(Z,\eta)
\right].
\end{equation*}
By the previous proposition,
\begin{equation*}
\|\mathcal T_p(Z)\|_{\mathcal X_p}
\le
B_K
\left[
M_K|\eta|
+
C_K^{(R)}
\left(
\|Z\|_{\mathcal X_p}^2
+
|\eta|\|Z\|_{\mathcal X_p}
\right)
\right].
\end{equation*}

Set
$
R_0:=2B_KM_K.
$
For $|\eta|>0$, define
\begin{equation*}
\mathcal B_\eta^p
:=
\left\{
Z\in\mathcal X_p:
\|Z\|_{\mathcal X_p}\le R_0|\eta|
\right\}.
\end{equation*}
We choose $\eta_K>0$ small enough such that $R_0\eta_K\le r_K,$
and $B_KC_K^{(R)}(R_0^2+R_0)\eta_K
\le
B_KM_K.$
Then, if $Z\in\mathcal B_\eta^p$ and $|\eta|\le\eta_K$, we have
\begin{equation*}
\begin{aligned}
\|\mathcal T_p(Z)\|_{\mathcal X_p}
&\le
B_KM_K|\eta|
+
B_KC_K^{(R)}
\left(
R_0^2|\eta|^2
+
R_0|\eta|^2
\right)
\\
&\le
2B_KM_K|\eta|
\\
&=
R_0|\eta|.
\end{aligned}
\end{equation*}
Thus $\mathcal T_p$ maps $\mathcal B_\eta^p$ into itself.

Next, for $Z_1,Z_2\in\mathcal B_\eta^p$, we have
\begin{equation*}
\begin{aligned}
\|\mathcal T_p(Z_1)-\mathcal T_p(Z_2)\|_{\mathcal X_p}
&\le
B_K
\|\widetilde R_p(Z_1,\eta)-\widetilde R_p(Z_2,\eta)\|_{\mathcal Y_p}
\\
&\le
B_KC_K^{(R)}
\left(
\|Z_1\|_{\mathcal X_p}
+
\|Z_2\|_{\mathcal X_p}
+
|\eta|
\right)
\|Z_1-Z_2\|_{\mathcal X_p}
\\
&\le
B_KC_K^{(R)}
(2R_0|\eta|+|\eta|)
\|Z_1-Z_2\|_{\mathcal X_p}.
\end{aligned}
\end{equation*}
We choose $\eta_K$ smaller if necessary such that
\begin{equation*}
B_KC_K^{(R)}(2R_0+1)\eta_K\le\frac12.
\end{equation*}
Then $\mathcal T_p$ is a contraction on $\mathcal B_\eta^p$.

By the contraction mapping theorem, there exists a unique fixed point
$
Z_p^\eta=(\widetilde\psi_p^\eta,\kappa_p^\eta)\in\mathcal B_\eta^p.
$
Therefore
\begin{equation*}
\|\widetilde\psi_p^\eta\|_{C^{2+\alpha,1+\alpha/2}}
+
|\kappa_p^\eta|
\le
R_0|\eta|.
\end{equation*}
This proves the proposition.
\end{proof}

Returning to the original variables, we set
\begin{equation*}
\psi_p^\eta(y,z)
:=
\widetilde\psi_p^\eta(y,z-p\cdot y).
\end{equation*}
Then define
\begin{equation*}
w_p^\eta:=w_p^0+\psi_p^\eta,
\qquad
c_\eta(p):=c_0(p)+\kappa_p^\eta,
\qquad
\chi_p^\eta(y,z):=z+w_p^\eta(y,z).
\end{equation*}
We next derive the corrector estimates needed in the homogenization argument.
\begin{lemma}[\textbf{Higher regularity of the corrector}]
Let $K\Subset\mathbb R^N$. Then there exists $C_K>0$ such that
\begin{equation*}
\|\chi_p^\eta-z\|_{C^2(\mathbb T^{N+1})}\le C_K
\end{equation*}
for all $p\in K$ and $|\eta|\le\eta_K$.
\end{lemma}

\begin{proof}
We write the corrector in the variables
$
r=z-p\cdot y,
$ $
\Xi(y,r):=\chi_p^\eta(y,p\cdot y+r).
$
Then
$
\chi_p^\eta(y,z)=\Xi(y,z-p\cdot y).
$
In the variables $(y,r)$, the cell equation becomes
\begin{equation*}
c_\eta(p)\Xi_r
+
F(\Xi,D_y\Xi)
+
\eta W(y,\Xi,D_y\Xi)
=
\Delta_y\Xi.
\end{equation*}
By the construction of the corrector, we already know that
\begin{equation*}
\Xi-\Xi^0\in C_{\rm per}^{2+\alpha,1+\alpha/2}(\mathbb T_p^{N+1}).
\end{equation*}
In particular,
$
\Xi,D_y\Xi, D_y^2\Xi, \Xi_r
$
are bounded.

We now differentiate the equation with respect to $r$. Set
$
U:=\Xi_r.
$
Then $U$ satisfies the linear periodic-parabolic equation
\begin{equation*}
c_\eta(p)U_r
+
\left[
F_s(\Xi,D_y\Xi)
+
\eta W_s(y,\Xi,D_y\Xi)
\right]U
+
\left[
D_pF(\Xi,D_y\Xi)
+
\eta D_pW(y,\Xi,D_y\Xi)
\right]\cdot D_yU
=
\Delta_yU.
\end{equation*}
The coefficients of this equation are uniformly Hölder continuous for $p\in K$ and $|\eta|\le\eta_K$. Moreover, since $c_0(p)<0$ and $c_\eta(p)=c_0(p)+O(\eta)$, by choosing $\eta_K$ smaller if necessary, we have
\begin{equation*}
|c_\eta(p)|\ge c_K>0
\end{equation*}
for all $p\in K$. Hence this is a uniformly parabolic periodic equation with respect to the time variable $r$.

By the periodic-parabolic Schauder estimate applied to the equation for $U$, we obtain $U=\Xi_r\in C_{\rm per}^{2+\alpha,1+\alpha/2}(\mathbb T_p^{N+1}),$
with a norm bounded by a constant depending only on $K$. Consequently,
$
U_r=\Xi_{rr}
$
is bounded, and
$
D_yU=D^2_{yr}\Xi
$
is bounded.

We now return to the variables $(y,z)$. Since
$
\chi_p^\eta(y,z)=\Xi(y,z-p\cdot y),
$
we have
$$
(\chi_p^\eta)_z=\Xi_r,
$$
$$
D_y\chi_p^\eta=D_y\Xi-p\Xi_r,
$$
and
\begin{equation*}
(\chi_p^\eta)_{zz}=\Xi_{rr}.
\end{equation*}
Moreover,
\begin{equation*}
D^2_{yz}\chi_p^\eta
=
D^2_{yr}\Xi-p\Xi_{rr},
\end{equation*}
and
\begin{equation*}
D^2_{yy}\chi_p^\eta
=
D^2_{yy}\Xi
-
p\otimes D_y\Xi_r
-
D_y\Xi_r\otimes p
+
p\otimes p\,\Xi_{rr}.
\end{equation*}
All terms on the right-hand sides are bounded uniformly for $p\in K$, because $p$ ranges in a compact set and we have already bounded
$
D_y^2\Xi,$ $ D_y\Xi_r,$ $ \Xi_{rr}.
$
Therefore $\|\chi_p^\eta-z\|_{C^2(\mathbb T^{N+1})}\le C_K.$
\end{proof}
\begin{remark}\label{rem:C2beta-corrector}
The proof of Lemma 3.2 actually gives a slightly stronger estimate. After decreasing the Hölder exponent if necessary, there exist
$
0<\beta<\frac{\alpha}{2}
$
and a constant $C_K>0$ such that
\begin{equation*}
\|\chi_p^\eta-z\|_{C^{2,\beta}(\mathbb T^{N+1})}\le C_K
\end{equation*}
for all $p\in K$ and $|\eta|\le\eta_K$.

Indeed, after setting
$
r=z-p\cdot y,
$ $
\Xi(y,r)=\chi_p^\eta(y,p\cdot y+r),
$
one first applies the uniform parabolic Schauder estimate on the skew torus to the equation for
$
Z(y,r):=\Xi(y,r)-p\cdot y-r.
$
Then one applies the same estimate to the differentiated equation for
$
U:=\Xi_r.
$
This gives Hölder bounds for
$
D_y^2Z,$ $ D_yZ_r,$ $ Z_{rr}
$
in the variables $(y,r)$. Since the change of variables
$
(y,z)\mapsto (y,z-p\cdot y)
$
is linear and $p$ ranges in the compact set $K$, these estimates imply the stated uniform $C^{2,\beta}$ estimate in the variables $(y,z)$.
\end{remark}
\begin{proposition}
Assume that $K$ is a compact subset of $\mathbb R^N$. Then there exist constants
$$
m_K,\ M_K,\ C_K,\ \eta_K>0
$$
such that
\begin{equation*}
m_K\le(\chi_p^\eta)_z(y,z)\le M_K
\end{equation*}
for $(y,z)\in\mathbb R^N\times\mathbb R$, and
\begin{equation*}
\|\chi_p^\eta-z\|_{C^2(\mathbb T^{N+1})}\le C_K
\end{equation*}
for $|\eta|\le\eta_K$ and $p\in K$.
\end{proposition}

\begin{proof}
We use the variables
$$
r=z-p\cdot y,
\qquad
\Xi_p^\eta(y,r):=\chi_p^\eta(y,p\cdot y+r).
$$
Similarly, set
$$
\Xi_p^0(y,r):=\chi_p^0(p\cdot y+r).
$$
By the construction in Proposition 3.3,
\begin{equation*}
\Xi_p^\eta-\Xi_p^0=\widetilde\psi_p^\eta,
\end{equation*}
and
\begin{equation*}
\|\widetilde\psi_p^\eta\|_{C_{\rm per}^{2+\alpha,1+\alpha/2}(\mathbb T_p^{N+1})}
+
|\kappa_p^\eta|
\le C_K|\eta|.
\end{equation*}
In particular,
\begin{equation*}
\|(\widetilde\psi_p^\eta)_r\|_{L^\infty}\le C_K|\eta|.
\end{equation*}

Since
$
(\chi_p^\eta)_z(y,z)
=
(\Xi_p^\eta)_r(y,z-p\cdot y),
$
and
$
(\Xi_p^0)_r(y,r)
=
(\chi_p^0)'(p\cdot y+r)
=
a_p(p\cdot y+r),
$
we have
\begin{equation*}
(\chi_p^\eta)_z(y,z)
=
a_p(z)+(\widetilde\psi_p^\eta)_r(y,z-p\cdot y).
\end{equation*}
By Proposition 2.3, there exist constants $m_K^0,M_K^0>0$ such that
\begin{equation*}
m_K^0\le a_p(z)\le M_K^0
\end{equation*}
for all $p\in K$ and $z\in\mathbb R$.

Choose $\eta_K>0$ smaller if necessary so that $C_K|\eta|\le\frac{m_K^0}{2}
\text{ for }|\eta|\le\eta_K.$

Then
\begin{equation*}
\frac{m_K^0}{2}
\le
(\chi_p^\eta)_z(y,z)
\le
M_K^0+\frac{m_K^0}{2}.
\end{equation*}
Thus we may take
$$
m_K:=\frac{m_K^0}{2},
\qquad
M_K:=M_K^0+\frac{m_K^0}{2}.
$$

Finally, by the bootstrap regularity lemma,
\begin{equation*}
\|\chi_p^\eta-z\|_{C^2(\mathbb T^{N+1})}\le C_K
\end{equation*}
for all $p\in K$ and $|\eta|\le\eta_K$. This completes the proof.
\end{proof}
We next prove the continuity of the effective speed.
\begin{proposition}
The map
$$
p\longmapsto c_\eta(p)
$$
is continuous on every compact set $K\Subset\mathbb R^N$ on which the above correctors are constructed.
\end{proposition}

\begin{proof}
Let
$
p_n\to p
\text{ in }K.
$
We prove that
$
c_\eta(p_n)\to c_\eta(p).
$
Write
$$
c_\eta(p_n)=c_0(p_n)+\kappa_{p_n}^\eta.
$$
Since $c_0$ is continuous, it is enough to prove
$
\kappa_{p_n}^\eta\to \kappa_p^\eta.
$

By Proposition \ref{estimate}, the sequence
$
\kappa_{p_n}^\eta
$
is bounded. Moreover, in the variables
$
r=z-p_n\cdot y,
$
the functions
$
\widetilde\psi_{p_n}^\eta
$
are uniformly bounded in the corresponding spaces
$$
C_{\rm per}^{2+\alpha,1+\alpha/2}(\mathbb T_{p_n}^{N+1}).
$$
Returning to the variables $(y,z)$ by
\begin{equation}
\psi_{p_n}^\eta(y,z)
=
\widetilde\psi_{p_n}^\eta(y,z-p_n\cdot y),
\end{equation}
and using the uniform corrector estimates and the bootstrap regularity lemma, we may extract a subsequence, still denoted by $p_n$, such that $\psi_{p_n}^\eta\to \psi_*$
locally uniformly together with the derivatives needed in the cell equation, and $\kappa_{p_n}^\eta\to\kappa_*.$
Since
$
p_n\to p,
$
and since the unperturbed correctors depend continuously on $p$, we have $w_{p_n}^0\to w_p^0$
in $C^2(\mathbb T)$. Hence
\begin{equation}
\chi_{p_n}^\eta(y,z)
=
z+w_{p_n}^0(z)+\psi_{p_n}^\eta(y,z)
\end{equation}
converges to
\begin{equation}
\chi_*(y,z):=z+w_p^0(z)+\psi_*(y,z).
\end{equation}
Passing to the limit in the cell equation for $\chi_{p_n}^\eta$, we obtain
\begin{equation*}
\begin{aligned}
(c_0(p)+\kappa_*)(\chi_*)_z
&+
F\left(\chi_*,D_y\chi_*+p(\chi_*)_z\right)
\\
&+
\eta W\left(y,\chi_*,D_y\chi_*+p(\chi_*)_z\right)
\\
&=
\Delta_y\chi_*
+
2p\cdot D^2_{yz}\chi_*
+
|p|^2(\chi_*)_{zz}.
\end{aligned}
\end{equation*}
Moreover, the normalization is preserved in the limit. Indeed, the change of variables
$$
(y,z)\longmapsto (y,z-p_n\cdot y)
$$
preserves the Lebesgue measure on the torus, and
$
\int_{\mathbb T_{p_n}^{N+1}}\widetilde\psi_{p_n}^\eta=0.
$
Therefore
\begin{equation*}
\int_{\mathbb T^{N+1}}\psi_*=0.
\end{equation*}

Thus the pair
$
(\psi_*,\kappa_*)
$
is a solution of (\ref{Cell}) with parameter $p$ and with the normalization
$
\int_{\mathbb T^{N+1}}\psi_*=0.
$
By the uniqueness of the small solution obtained in Proposition \ref{estimate}, we must have
\begin{equation*}
\psi_*=\psi_p^\eta,
\qquad
\kappa_*=\kappa_p^\eta.
\end{equation*}
Consequently, every convergent subsequence of $\kappa_{p_n}^\eta$ has the same limit $\kappa_p^\eta$. Hence
\begin{equation*}
\kappa_{p_n}^\eta\to\kappa_p^\eta.
\end{equation*}
Therefore
\begin{equation*}
c_\eta(p_n)\to c_\eta(p).
\end{equation*}
This proves the continuity of $c_\eta$ on $K$.
\end{proof}
By Lemma~\ref{lem:initial-trace}, the family $\{u^\epsilon\}$ is locally uniformly bounded on finite time intervals, and the half-relaxed limits satisfy the initial condition. We also record a spatial Lipschitz estimate for the half-relaxed limits.

\begin{lemma}[\textbf{Spatial Lipschitz estimates for the half-relaxed limits}]
 
Denote
\begin{equation*}
u^*:={\limsup}_{\epsilon\to0}^{*}u^\epsilon,
\qquad
u_*:={\liminf}_{\epsilon\to0,*}u^\epsilon.
\end{equation*}
Then there exists a constant $L>0$ such that
\begin{equation*}
|u^*(x+h,t)-u^*(x,t)|\le L|h|,
\end{equation*}
and
\begin{equation*}
|u_*(x+h,t)-u_*(x,t)|\le L|h|,
\end{equation*}
for every $x,h\in\mathbb R^N$ and $t\ge0$. \label{Lipschitz}
\end{lemma}
\begin{proof}
 For $k\in\mathbb Z^N$, the function
$
u^\epsilon(x+k\epsilon,t)
$
also solves (\ref{HJ}), with initial data
$
u_0(x+k\epsilon).
$
Since $u_0$ is Lipschitz continuous, there exists a constant $L_0$ such that
\begin{equation*}
u_0(x+k\epsilon)\le u_0(x)+\epsilon L_0|k|.
\end{equation*}
By the comparison principle and the periodicity in the $u^\epsilon/\epsilon$ variable,
\begin{equation*}
u^\epsilon(x+k\epsilon,t)
\le
u^\epsilon(x,t)+\epsilon\lceil L_0|k|\rceil
\le
u^\epsilon(x,t)+\epsilon L_0|k|+\epsilon.
\end{equation*}
Pick a sequence $k_\epsilon\in\mathbb Z^N$ such that
$$
\epsilon k_\epsilon\to h
\quad\text{as }\epsilon\to0.
$$
Then
\begin{equation*}
u^*(x+h,t)\le u^*(x,t)+L_0|h|,
\end{equation*}
and
\begin{equation*}
u_*(x+h,t)\le u_*(x,t)+L_0|h|.
\end{equation*}
Moreover, by replacing $h$ with $-h$, we obtain
\begin{equation*}
|u^*(x+h,t)-u^*(x,t)|\le L_0|h|,
\end{equation*}
and
\begin{equation*}
|u_*(x+h,t)-u_*(x,t)|\le L_0|h|.
\end{equation*}
This proves the lemma.
\end{proof}

\subsection{Homogenization for the perturbative case}

We now prove the homogenization theorem for the perturbed problem.
\begin{result}
Consider
\begin{equation}
\begin{cases}
u_t^\epsilon
+
F\left(\dfrac{u^\epsilon}{\epsilon},Du^\epsilon\right)
+
\eta W\left(
\dfrac{x}{\epsilon},
\dfrac{u^\epsilon}{\epsilon},
Du^\epsilon
\right)
=
\epsilon\Delta u^\epsilon
&\text{in }\mathbb R^N\times(0,\infty),
\\[2mm]
u^\epsilon(x,0)=u_0(x)
&\text{on }\mathbb R^N.
\end{cases}
\end{equation}
Suppose that
$$
u_0\in BUC(\mathbb R^N)\cap W^{1,\infty}(\mathbb R^N).
$$
Let $K_0:=\overline{B_{\operatorname{Lip}(u_0)}(0)}.$
Assume that $F$ satisfies \textnormal{(F1)}--\textnormal{(F4)}, that
\[
W\in C_{\mathrm{loc}}^{3,\alpha}
(\mathbb T^N\times\mathbb T\times\mathbb R^N),
\]
and that there exists $L>0$ such that
\[
|W(x,r,p)-W(y,s,q)|
\leq
L\big(
(1+|p|+|q|)|x-y|
+|r-s|
+|p-q|
\big)
\]
for all $x,y\in\mathbb R^N$, $r,s\in\mathbb R$, and
$p,q\in\mathbb R^N$. Then there exists $\eta_{K_0}>0$ such that, if
\begin{equation}
|\eta|\le\eta_{K_0},
\end{equation}
then
$$
u^\epsilon\to u
$$
locally uniformly, where $u$ is the viscosity solution of
\begin{equation}
\begin{cases}
u_t-c_\eta(Du)=0
&\text{in }\mathbb R^N\times(0,\infty),
\\
u(x,0)=u_0(x)
&\text{on }\mathbb R^N.
\end{cases}
\end{equation}
\end{result}
\begin{proof}
 The proof is based on the unperturbed argument, with the corrector replacing the one-dimensional profile. We use the following perturbed test function:
\begin{equation*}
\Phi_{s_\epsilon}^\epsilon(x,t)
=
\epsilon
\chi_p^\eta\left(
\frac{x}{\epsilon},
\frac{\phi(x,t)}{\epsilon}+s_\epsilon
\right).
\end{equation*}
We again test the upper half-relaxed limit $u^*$. Suppose that $(x_0,t_0)$ is the maximum point of $u^*-\phi$. Set
$$
a:=\phi_t(x_0,t_0),
\qquad
p:=D\phi(x_0,t_0).
$$
By Lemma \ref{Lipschitz}, the half-relaxed limits are spatially Lipschitz with Lipschitz constant not larger than $L_0:=\operatorname{Lip}(u_0)$. Therefore every gradient $p$ appearing in the viscosity test belongs to
$
K_0=\overline{B_{L_0}(0)}.
$
We apply the corrector construction in Section 3 with this compact set $K_0$.
By Lemma \ref{phase-shift contact lemma}, applied to
$$
\chi(y,z)=\chi_p^\eta(y,z),
$$
there exist $s_\epsilon$ and maximum points $(x_\epsilon,t_\epsilon)\to(x_0,t_0)$ such that $M_\epsilon(s_\epsilon)=0,$
and
\begin{equation*}
\frac{u^\epsilon(x_\epsilon,t_\epsilon)}{\epsilon}
=
\chi_p^\eta\left(
\frac{x_\epsilon}{\epsilon},
\frac{\phi(x_\epsilon,t_\epsilon)}{\epsilon}+s_\epsilon
\right).
\end{equation*}
Then $u^\epsilon-\Phi_{s_\epsilon}^\epsilon$ attains its local maximum at $(x_\epsilon,t_\epsilon)$, and
$
(x_\epsilon,t_\epsilon)\to(x_0,t_0).
$
By the viscosity subsolution test and the cell equation, we obtain
\begin{align*}
0\ge\;&
(\chi_p^\eta)_z\left(\phi_t-c_\eta(p)\right)
\\
&+
\left[
F\left(\chi_p^\eta,D_y\chi_p^\eta+(\chi_p^\eta)_zD\phi\right)
-
F\left(\chi_p^\eta,D_y\chi_p^\eta+(\chi_p^\eta)_zp\right)
\right]
\\
&+
\eta
\left[
W\left(y,\chi_p^\eta,D_y\chi_p^\eta+(\chi_p^\eta)_zD\phi\right)
-
W\left(y,\chi_p^\eta,D_y\chi_p^\eta+(\chi_p^\eta)_zp\right)
\right]
\\
&-
2(D\phi-p)\cdot D^2_{yz}\chi_p^\eta
-
\left(|D\phi|^2-|p|^2\right)(\chi_p^\eta)_{zz}
-
\epsilon(\chi_p^\eta)_z\Delta\phi.
\end{align*}
Letting $\epsilon\to0$, we get
\begin{equation*}
0\ge(\chi_p^\eta)_z(a-c_\eta(p)).
\end{equation*}
If
$$
a>c_\eta(p),
$$
then by the lower bound of $(\chi_p^\eta)_z$,
$
0\ge m_K(a-c_\eta(p))>0,
$
which is a contradiction. Therefore
$$
a\le c_\eta(p).
$$
Thus $u^*$ is a viscosity subsolution.

Similarly, $u_*$ is a viscosity supersolution. By the comparison principle,
$
u^*\le u_*.
$
Since
$
u_*\le u^*,
$
we conclude that
$
u:=u^*=u_*
$
is the viscosity solution of the effective equation, and
$$
u^\epsilon\to u
$$
locally uniformly.
\end{proof}

\section{Large-time average in the general case}

For a general contact-type Hamiltonian, we do not construct correctors in this paper. Nevertheless, the associated cell evolution has a well-defined large-time average under the assumptions below, following the viewpoint of \cite{imbert2008homogenization}.
We impose the following assumptions on the Hamiltonian $H$.

\begin{itemize}
 \item[(H1).] \textbf{Regularity:} The Hamiltonian
\[
H:\mathbb R^N\times\mathbb R\times\mathbb R^N\to\mathbb R
\]
is Lipschitz continuous, and there exists $\gamma\geq0$ such that, for almost every $(y,v,p)\in\mathbb R^N\times\mathbb R\times\mathbb R^N$,
\[
|D_yH(y,v,p)|\leq\gamma(1+|p|),\qquad
|\partial_vH(y,v,p)|\leq\gamma,\qquad
|D_pH(y,v,p)|\leq\gamma.
\]

 \item[(H2).] \textbf{Periodicity:} for any $(y,v, p) \in \mathbb{R}^N \times \mathbb{R} \times \mathbb{R}^N$:
 \[
 H(y+k,v+l, p) = H(y,v, p) \quad \text{for any } l \in \mathbb{Z}, k \in \mathbb{Z}^N;
 \]
\end{itemize}
\subsection{The cell equation from an embedded viewpoint}
For a general contact-type Hamiltonian, the corrector problems constructed in Sections 2 and 3 are not available. We therefore return to the higher-dimensional formulation introduced by Imbert and Monneau~\cite{imbert2008homogenization}.

As explained in \cite{imbert2008homogenization}, the usual ansatz is obstructed by an uncontrolled phase. We first present the formal computation. Suppose $u^\epsilon(x,t) \approx u(x,t) + \epsilon v(y, \tau)$, where $y = \frac{x}{\epsilon}, \tau = \frac{t}{\epsilon}$. Here $\epsilon > 0$ is small. Then
\[
u^\epsilon_t(x,t) \approx u_t(x,t) + v_\tau(y,\tau),
\]
\[
D_x u^\epsilon(x,t) \approx D_x u(x,t) + D_y v(y,\tau),
\]
\[
\Delta_x u^\epsilon(x,t) \approx \Delta_x u(x,t) + \frac{1}{\epsilon}\Delta_y v(y,\tau).
\]
Let $\lambda = u_t(0,0)$ and $p = D_x u(0,0)$. Using the Taylor expansion $u(x,t) \approx u(0,0)+\lambda t+p\cdot x$ and substituting the formal ansatz into (1.1), we obtain
\[
\lambda + v_\tau(y,\tau) + H\left(y, \frac{u(0,0)}{\epsilon} + \lambda \tau + py, p + D_y v(y,\tau)\right) \approx \epsilon \Delta_x u(0,0) + \Delta_y v(y,\tau).
\]
However, when we let $\epsilon \rightarrow 0$, $\frac{u(0,0)}{\epsilon}$ cannot be controlled if $u(0,0) \neq 0$. Therefore the usual ansatz does not yield an appropriate corrector.

Imbert and Monneau introduced an embedding method in one additional space variable. Let $U^\epsilon(x,x_{N+1},t)$ be the viscosity solution of
\begin{equation}
\begin{cases}
U^\epsilon_t + H\left(\frac{x}{\epsilon}, \frac{U^\epsilon}{\epsilon}, D_x U^\epsilon\right) = \epsilon \Delta_x U^\epsilon & \text{in } \mathbb{R}^{N+1} \times (0,\infty), \\
U^\epsilon(x, x_{N+1}, 0) = u_0(x) + x_{N+1} & \text{on } \mathbb{R}^{N+1}.
\end{cases}
\label{eq:5.1}
\end{equation}

\vspace{1cm}

Assume formally that $U^\epsilon$ converges locally uniformly to $U$. Then $U(x, x_{N+1}, 0) = u_0(x) + x_{N+1}$, and $U(X,t) = u(x,t) + x_{N+1}$. We now consider the following ansatz:
\begin{equation*}
U^\epsilon(x, x_{N+1}, t) = U(x, x_{N+1}, t) + \epsilon V\left(\frac{x}{\epsilon}, \frac{U(t, x, x_{N+1}) - \lambda t - px}{\epsilon}, \frac{t}{\epsilon}\right),
\label{eq:5.2}
\end{equation*}
where $\lambda = U_t(0,0,0)$ and $p = D_x U(0,0,0)$. A formal substitution gives
\begin{align*}
 U_t + V_\tau + V_{y_{N+1}} \cdot (U_t - \lambda) &+ H\left(y, V + \lambda \tau + py + y_{N+1}, p + D_y V+ \frac{\partial V}{\partial y_{N+1}}\cdot (D_x U - p)\right) \\
& = \Delta_y V + \frac{\partial^2 V}{\partial y_{N+1}^2} |D_x U - p|^2 + \epsilon \frac{\partial V}{\partial y_{N+1}} \Delta U.
\end{align*}
\bigskip
Here $y_{N+1} := \frac{U(t, x, x_{N+1}) - \lambda t - px}{\epsilon}$. Letting $\epsilon\to0$ and $(x,x_{N+1},t)\to(0,0,0)$ formally yields
\begin{equation*}
\lambda + V_\tau + H(y, V + \lambda \tau + py + y_{N+1}, p + D_y V) = \Delta_y V.
\label{eq:5.3}
\end{equation*}
This motivates the following cell evolution:
\begin{equation*}
\begin{cases}
w_\tau+H(\zeta,p\cdot\zeta+w,p+D_\zeta w)=\Delta_\zeta w,
&(\tau,\zeta)\in(0,\infty)\times\mathbb R^N,\\
w(\zeta,0)=0,
&\zeta\in\mathbb R^N.
\end{cases}
\end{equation*}
This evolution will be used to define the ergodic constant.

\subsection{The ergodic constant}
Following Barles and Souganidis~\cite{barles2001space}, we first establish an oscillation estimate for the cell evolution.
Assumption (H1) implies that 
\begin{equation*}
|H(\zeta,s,q)|\le \chi(|q|)
\quad\text{for all }(\zeta,s,q),
\qquad
\int^\infty \frac{dr}{\chi(r)}=+\infty.
\end{equation*}
In the present setting we may take $\chi$ to be linear.

\begin{proposition}[\textbf{Oscillation estimate}]
Assume that $H$ satisfies (H1)(H2). Consider
\begin{equation*}
\begin{cases}
w_\tau+H(\zeta,p\cdot\zeta+w,p+D_\zeta w)=\Delta_\zeta w,
&(\tau,\zeta)\in(0,\infty)\times\mathbb R^N,\\
w(\zeta,0)=0,
&\zeta\in\mathbb R^N.
\end{cases}
\end{equation*}
Then there exists a constant $C_p>0$, such that
\begin{equation*}
\operatorname{osc}_{\mathbb R^N} w(\cdot,\tau)\le C_p, \quad \text{for each}\quad \tau \ge 0.
\end{equation*}\label{oscillation}
\end{proposition}

\begin{proof}
Suppose $R>\sqrt N$ and take a function
$
\psi\in C^2((0,R])\cap C([0,R])
$
such that
$
\psi(0)=0,$ $ \psi'>0,$ $ \psi''<0,
$
and
\begin{equation*}
4\psi''(r)+2\chi(|p|+\psi'(r))<0
\quad\text{for }0<r<R.
\end{equation*}
Let
$
z(r)=\psi'(r).
$
Consider
\begin{equation*}
4z'(r)=-3\chi(|p|+z(r)).
\end{equation*}
Since
$
\int^\infty \frac{dr}{\chi(r)}=+\infty,
$
we choose $z(0)$ sufficiently large so that
$$
z(r)>0\quad\text{on }[0,R).
$$
Then
\begin{equation*}
4\psi''+2\chi(|p|+\psi')
=
-\chi(|p|+\psi')<0.
\end{equation*}

\medskip

\noindent\textbf{Local Lipschitz estimate.}
We show that
\begin{equation*}
w(\zeta,\tau)-w(\eta,\tau)\le \psi(|\zeta-\eta|)
\quad\text{for }|\zeta-\eta|\le R.
\end{equation*}
Suppose that
\begin{equation*}
\max_{[0,\infty)\times\mathbb R^{2N}}
\left\{
w(\zeta,\tau)-w(\eta,\tau)-\psi(|\zeta-\eta|)
\right\}>0.
\end{equation*}
Consider
\begin{equation*}
\Phi(\zeta,\eta,\tau)
=
w(\zeta,\tau)-w(\eta,\tau)-\psi(|\zeta-\eta|)
-\alpha(|\zeta|^2+|\eta|^2)
-\frac{\rho}{T-\tau}.
\end{equation*}
Assume that
\begin{equation*}
\max_{[0,T]\times\{|\zeta-\eta|\le R\}}\Phi(\zeta,\eta,\tau)>0.
\end{equation*}
Then $\Phi$ attains its maximum at
$
(\tau_\alpha,\zeta_\alpha,\eta_\alpha).
$
At this maximum point,
$
\tau_\alpha>0,\qquad \zeta_\alpha\ne\eta_\alpha.
$
By Ishii's lemma~\cite{CrandallIshiiLions1992}, there exist $a\in\mathbb R$ and $X,Y\in\mathbb S^N$ such that
\begin{equation*}
\left(
a+\frac{\rho}{(T-\tau_\alpha)^2},
\psi'(r_\alpha)e_\alpha+2\alpha\zeta_\alpha,
X
\right)
\in
\overline J^{2,+}w(\zeta_\alpha,\tau_\alpha),
\end{equation*}
\begin{equation*}
\left(
a,
\psi'(r_\alpha)e_\alpha-2\alpha\eta_\alpha,
Y
\right)
\in
\overline J^{2,-}w(\eta_\alpha,\tau_\alpha),
\end{equation*}
and
\begin{equation*}
\operatorname{tr}X-\operatorname{tr}Y
\le
4\psi''(r_\alpha)+C\alpha,
\end{equation*}
where
$
r_\alpha:=|\zeta_\alpha-\eta_\alpha|,
$ $
e_\alpha:=\frac{\zeta_\alpha-\eta_\alpha}{|\zeta_\alpha-\eta_\alpha|}.
$

Using the subsolution and supersolution inequalities,
\begin{equation*}
a+\frac{\rho}{(T-\tau_\alpha)^2}
+
H\left(
\zeta_\alpha,
p\cdot\zeta_\alpha+w(\zeta_\alpha,\tau_\alpha),
p+\psi'(r_\alpha)e_\alpha+2\alpha\zeta_\alpha
\right)
\le
\operatorname{tr}X,
\end{equation*}
and
\begin{equation*}
a+
H\left(
\eta_\alpha,
p\cdot\eta_\alpha+w(\eta_\alpha,\tau_\alpha),
p+\psi'(r_\alpha)e_\alpha-2\alpha\eta_\alpha
\right)
\ge
\operatorname{tr}Y.
\end{equation*}
Therefore
\begin{align*}
&\frac{\rho}{(T-\tau_\alpha)^2}
+
H\left(
\zeta_\alpha,
p\cdot\zeta_\alpha+w(\zeta_\alpha,\tau_\alpha),
p+\psi'(r_\alpha)e_\alpha+2\alpha\zeta_\alpha
\right)
\\
&\quad
-
H\left(
\eta_\alpha,
p\cdot\eta_\alpha+w(\eta_\alpha,\tau_\alpha),
p+\psi'(r_\alpha)e_\alpha-2\alpha\eta_\alpha
\right)
\\
&\le
\operatorname{tr}X-\operatorname{tr}Y.
\end{align*}
Since
\begin{align*}
&H\left(
\zeta_\alpha,
p\cdot\zeta_\alpha+w(\zeta_\alpha,\tau_\alpha),
p+\psi'(r_\alpha)e_\alpha+2\alpha\zeta_\alpha
\right)
\\
&\quad
-
H\left(
\eta_\alpha,
p\cdot\eta_\alpha+w(\eta_\alpha,\tau_\alpha),
p+\psi'(r_\alpha)e_\alpha-2\alpha\eta_\alpha
\right)
\\
&\ge
-2\chi\left(|p|+\psi'(r_\alpha)+O_\alpha(1)\right),
\end{align*}
letting $\alpha\to0$, we have
\begin{equation*}
0\le 4\psi''(r)+2\chi(|p|+\psi'(r)),
\end{equation*}
which is a contradiction. Therefore
\begin{equation*}
|w(\zeta,\tau)-w(\eta,\tau)|
\le
\psi(|\zeta-\eta|)
\quad\text{when }|\zeta-\eta|\le R.
\end{equation*}

\medskip

\noindent\textbf{Global Lipschitz estimate.}
Pick $\zeta,\eta\in\mathbb R^N$ and $k\in\mathbb Z^N$ such that
\begin{equation*}
|\zeta-(\eta+k)|\le \sqrt N<R.
\end{equation*}
Moreover, by the periodicity in the second variable of $H$,
\begin{equation*}
w(\eta+k,\tau)+p\cdot k-\lfloor p\cdot k\rfloor
\le
w(\eta,\tau)+1.
\end{equation*}
Therefore
\begin{align*}
w(\zeta,\tau)-w(\eta,\tau)
&\le
w(\zeta,\tau)-w(\eta+k,\tau)+1 \le
C_p.
\end{align*}
Finally,
\begin{equation*}
\operatorname{osc}_{\mathbb R^N}w(\cdot,\tau)\le C_p
\quad\text{for all }\tau\ge0.
\end{equation*}
\end{proof}

The oscillation estimate yields the following large-time behavior.

\begin{result}\label{large time}
Assume that $H$ satisfies (H1)(H2). Consider
\begin{equation*}
\begin{cases}
w_\tau+H(\zeta,p\cdot\zeta+w,p+Dw)=\Delta w
&\text{in }\mathbb R^N \times(0,\infty),\\
w(\zeta,0)=0
&\text{on }\mathbb R^N.
\end{cases}
\end{equation*}
Then
\begin{equation*}
\frac{w(\zeta,\tau)}{\tau}\to \lambda(p)
\quad\text{as }\tau\to+\infty \quad \text{locally uniformly}.
\end{equation*}
Moreover, there exists a constant $C>0$, independent of $\zeta$ and $\tau$, such that
\[
|w(\zeta,\tau)-\lambda(p)\tau|\le C
\]
for all $(\zeta,\tau)\in\mathbb R^N\times(0,\infty)$.
\end{result}

Before we prove our result, we need to introduce a lemma.

\begin{lemma}[\textbf{Almost additive functions}]
Let $f:[0,\infty)\to\mathbb R$ be locally bounded. Assume that there exists a constant $C_0>0$ such that
\begin{equation*}
|f(t+s)-f(t)-f(s)|\le C_0
\qquad\text{for all }t,s\ge0.
\end{equation*}
Then there exists a constant $\lambda\in\mathbb R$ such that
\begin{equation*}
\lim_{t\to+\infty}\frac{f(t)}{t}=\lambda.
\end{equation*}
Moreover,
\begin{equation*}
|f(t)-\lambda t|\le C_0
\qquad\text{for all }t\ge0.
\end{equation*} \label{almost additive}
\end{lemma}

\begin{proof}
For each $t\ge0$, define
\begin{equation*}
A(t):=\lim_{n\to\infty}2^{-n}f(2^n t),
\end{equation*}
provided that the limit exists. We first show that this limit is well-defined.

For $n\ge0$, by the almost-additivity assumption with $s=2^n t$, we have $\left|f(2^{n+1}t)-2f(2^n t)\right|\le C_0.$
Hence
\begin{align*}
\left|
2^{-(n+1)}f(2^{n+1}t)-2^{-n}f(2^n t)
\right|
&=
2^{-(n+1)}
\left|
f(2^{n+1}t)-2f(2^n t)
\right|
\\
&\le
2^{-(n+1)}C_0.
\end{align*}
Therefore the sequence
$
\left\{2^{-n}f(2^n t)\right\}_{n\ge0}
$
is Cauchy. Thus $A(t)$ is well-defined for every $t\ge0$. Moreover,
\begin{align*}
|A(t)-f(t)|
&\le
\sum_{n=0}^{\infty}
\left|
2^{-(n+1)}f(2^{n+1}t)-2^{-n}f(2^n t)
\right|
\\
&\le
\sum_{n=0}^{\infty}2^{-(n+1)}C_0
=
C_0.
\end{align*}
Thus
\begin{equation*}
|A(t)-f(t)|\le C_0
\qquad\text{for all }t\ge0.
\end{equation*}

Next, we prove that $A$ is additive. For $t,s\ge0$, by the almost-additivity assumption,
\begin{equation*}
\left|
f(2^n(t+s))-f(2^n t)-f(2^n s)
\right|
\le C_0.
\end{equation*}
Multiplying by $2^{-n}$ and letting $n\to\infty$, we obtain
\begin{equation*}
A(t+s)=A(t)+A(s)
\qquad\text{for all }t,s\ge0.
\end{equation*}

Since $f$ is locally bounded and $|A-f|\le C_0$, the function $A$ is also locally bounded. A locally bounded additive function on $[0,\infty)$ is linear. Indeed, if we set
$
\lambda:=A(1),
$
then additivity gives
$
A(q)=\lambda q
\text{ for all }q\in\mathbb Q_+.
$
Local boundedness implies continuity of $A$ at $0$. Therefore, by density of $\mathbb Q_+$ in $[0,\infty)$,
\begin{equation*}
A(t)=\lambda t
\qquad\text{for all }t\ge0.
\end{equation*}

Combining this with $|A(t)-f(t)|\le C_0$, we obtain
\begin{equation*}
|f(t)-\lambda t|\le C_0
\qquad\text{for all }t\ge0.
\end{equation*}
In particular,
\begin{equation*}
\lim_{t\to+\infty}\frac{f(t)}{t}=\lambda.
\end{equation*}
This completes the proof.
\end{proof}

We now complete the proof of Theorem \ref{large time}.

\begin{proof}[Proof of Theorem \ref{large time}]
Let
$$
m(\tau):=\inf_{\zeta\in\mathbb R^N}w(\zeta,\tau),
\qquad
M(\tau):=\sup_{\zeta\in\mathbb R^N}w(\zeta,\tau).
$$
Then by Proposition \ref{oscillation},
$$
M(\tau)-m(\tau)\le C_p,
$$
and
$$
m(s)\le w(\zeta,s)\le m(s)+C_p.
$$
Therefore there exist an integer $N_p$ and
$
n_s=\lfloor m(s)\rfloor
$
such that
\begin{equation*}
n_s\le w(\zeta,s)\le n_s+N_p
\quad\text{for all }\zeta\in\mathbb R^N.
\end{equation*}
By the comparison principle, we have
\begin{equation*}
w(\zeta,\tau)+n_s
\le
w(\zeta,s+\tau)
\le
w(\zeta,\tau)+n_s+N_p.
\end{equation*}
Taking the infimum with respect to $\zeta$, we obtain
\begin{equation*}
m(\tau)+n_s
\le
m(s+\tau)
\le
m(\tau)+n_s+N_p.
\end{equation*}
Moreover,
\begin{equation*}
m(\tau)+m(s)-1
\le
m(s+\tau)
\le
m(\tau)+m(s)+N_p.
\end{equation*}
Therefore $m$ is an almost-additive function. By Lemma \ref{almost additive}, there exists $\lambda(p)$ such that
\begin{equation*}
|w(\zeta,\tau)-\lambda(p)\tau|\le C.
\end{equation*}
\end{proof}

\subsection{Identification of the cell speed with the ergodic constant}

We now identify the corrector speed obtained from (\ref{Cell}) with the ergodic constant defined by the large-time average in Theorem 6.

Suppose that $H(y,s,p)=F(s,p)+\eta W(y,s,p)$, and we assume that $F$ satisfies \textup{(F1)}--\textup{(F4)}, and
$$
W\in C^{3,\alpha}(\mathbb T^N\times\mathbb T\times\mathbb R^N).
$$
Assume that, for each $p\in\mathbb R^N$, there exist $c_\eta(p)\in\mathbb R$ and a corrector $\chi_p^\eta$ satisfying
\[
\chi_p^\eta(y+k,z)=\chi_p^\eta(y,z),
\qquad
\chi_p^\eta(y,z+1)=\chi_p^\eta(y,z)+1,
\]
and
\begin{equation*}
\begin{aligned}
c_\eta(p)(\chi_p^\eta)_z
&+
F\left(\chi_p^\eta,D_y\chi_p^\eta+p(\chi_p^\eta)_z\right)
\\
&+
\eta W\left(y,\chi_p^\eta,D_y\chi_p^\eta+p(\chi_p^\eta)_z\right)
\\
&=
\Delta_y\chi_p^\eta
+
2p\cdot D^2_{yz}\chi_p^\eta
+
|p|^2(\chi_p^\eta)_{zz}.
\end{aligned}
\end{equation*}
For fixed $p\in\mathbb R^N$, consider the cell evolution
\begin{equation*}
\begin{cases}
w_\tau
+
F\left(p\cdot \zeta+w,p+D_\zeta w\right)
\\
\qquad
+
\eta W\left(\zeta,p\cdot \zeta+w,p+D_\zeta w\right)
=
\Delta_\zeta w,
&(\tau,\zeta)\in(0,\infty)\times\mathbb R^N,
\\[1mm]
w(\zeta,0)=0,
&\zeta\in\mathbb R^N.
\end{cases}
\end{equation*}

\begin{result}
For every fixed $p\in\mathbb R^N$, the solution $w=w(\zeta,\tau)$ of the above cell evolution satisfies
\begin{equation*}
\frac{w(\zeta,\tau)}{\tau}\to c_\eta(p)
\quad\text{as }\tau\to+\infty,
\end{equation*}
locally uniformly in $\zeta$. In particular, the ergodic constant defined by the long-time average is equal to the cell speed $c_\eta(p)$.
\end{result}

\begin{proof}
Define
\begin{equation*}
\widetilde w(\zeta,\tau)
:=
\chi_p^\eta\left(\zeta,p\cdot\zeta+c_\eta(p)\tau\right)
-
p\cdot\zeta.
\end{equation*}
We first verify that $\widetilde w$ solves the same cell evolution.

Set $z=p\cdot\zeta+c_\eta(p)\tau.$
Then $\widetilde w_\tau
=
c_\eta(p)(\chi_p^\eta)_z,$
and $p+D_\zeta\widetilde w
=
D_y\chi_p^\eta+p(\chi_p^\eta)_z.$
Moreover, $p\cdot\zeta+\widetilde w
=
\chi_p^\eta(\zeta,z).$
Finally,
\[
\Delta_\zeta \widetilde w
=
\Delta_y\chi_p^\eta
+
2p\cdot D^2_{yz}\chi_p^\eta
+
|p|^2(\chi_p^\eta)_{zz}.
\]
Therefore, by the cell equation satisfied by $\chi_p^\eta$, we obtain
\[
\begin{aligned}
\widetilde w_\tau
&+
F\left(p\cdot \zeta+\widetilde w,p+D_\zeta\widetilde w\right)
\\
&+
\eta W\left(\zeta,p\cdot \zeta+\widetilde w,p+D_\zeta\widetilde w\right)
=
\Delta_\zeta \widetilde w.
\end{aligned}
\]
Thus $\widetilde w$ is a solution of the cell evolution, with initial data
\begin{equation*}
\widetilde w(\zeta,0)
=
\chi_p^\eta(\zeta,p\cdot\zeta)-p\cdot\zeta.
\end{equation*}

Since $\chi_p^\eta(y,z)=z+w_p^\eta(y,z)$
with $w_p^\eta$ periodic in $(y,z)$, the function $\widetilde w(\zeta,0)$ is bounded. Hence there exists an integer $M_p\in\mathbb N$ such that
\begin{equation*}
-M_p\le \widetilde w(\zeta,0)\le M_p
\quad\text{for all }\zeta\in\mathbb R^N.
\end{equation*}

Let $w$ be the solution of the cell evolution with initial data $0$. Since $F$ and $W$ are periodic in the phase variable, if $u$ is a solution, then $u+m$ is also a solution for every $m\in\mathbb Z$. Therefore $w-M_p$ and $w+M_p$ are also solutions of the same cell equation. At $\tau=0$, we have
\[
w(\zeta,0)-M_p
\le
\widetilde w(\zeta,0)
\le
w(\zeta,0)+M_p.
\]
By the comparison principle,
\begin{equation*}
w(\zeta,\tau)-M_p
\le
\widetilde w(\zeta,\tau)
\le
w(\zeta,\tau)+M_p
\end{equation*}
for all $\tau\ge0$ and $\zeta\in\mathbb R^N$. Hence
\begin{equation*}
|w(\zeta,\tau)-\widetilde w(\zeta,\tau)|\le M_p. \label{4.48}
\end{equation*}

On the other hand, since $\chi_p^\eta(y,z)=z+w_p^\eta(y,z),$
we have $\widetilde w(\zeta,\tau)
=
c_\eta(p)\tau
+
w_p^\eta\left(\zeta,p\cdot\zeta+c_\eta(p)\tau\right).$
The function $w_p^\eta$ is bounded. Therefore
\begin{equation*}
\left|
\frac{\widetilde w(\zeta,\tau)}{\tau}
-
c_\eta(p)
\right|
\le
\frac{\|w_p^\eta\|_{L^\infty}}{\tau}.
\end{equation*}
Combining this estimate with (\ref{4.48}),
we obtain
\begin{equation*}
\left|
\frac{w(\zeta,\tau)}{\tau}
-
c_\eta(p)
\right|
\le
\frac{M_p+\|w_p^\eta\|_{L^\infty}}{\tau}.
\end{equation*}
Letting $\tau\to+\infty$, we conclude that
\begin{equation*}
\frac{w(\zeta,\tau)}{\tau}\to c_\eta(p).
\end{equation*}
The convergence is locally uniform in $\zeta$, and in fact uniform in $\zeta$.
\end{proof}

\section{Examples}
In this section we provide examples satisfying the assumptions above.
\begin{example}[Semilinear heat equation]
 Let 
 \begin{equation*}
 F(s,q)=g(s), \quad s\in \mathbb R,
 \end{equation*}
 where $g\in C^{3,\alpha}(\mathbb T)$ and has a positive lower bound. Then $F$ satisfies \textup{(F1)--(F4)}.
\end{example}
\begin{example}[An admissible unperturbed Hamiltonian]
Let
\begin{equation*}
F(s,q)
=
4+\sqrt{1+|q|^2}
+
\frac12\,\frac{\sin(2\pi s)}{\sqrt{1+|q|^2}},
\qquad
(s,q)\in\mathbb R\times\mathbb R^N.
\end{equation*}
Then $F$ satisfies \textup{(F1)--(F4)}.

Indeed, $F$ is $1$-periodic in $s$ and is smooth in $(s,q)$. Moreover,
\begin{equation*}
F(s,q)
\ge
4+1-\frac12
=
\frac92>0.
\end{equation*}
Thus \textup{(F1)} holds.

Next, we compute the recession function. For $\rho>0$,
\begin{align*}
\frac{F(s,\rho q)}{\rho}
&=
\frac4\rho
+
\frac{\sqrt{1+\rho^2|q|^2}}{\rho}
+
\frac{1}{2\rho}\frac{\sin(2\pi s)}{\sqrt{1+\rho^2|q|^2}}.
\end{align*}
Since
$$
\frac{\sqrt{1+\rho^2|q|^2}}{\rho}
=
\sqrt{|q|^2+\frac1{\rho^2}}
\to |q|
$$
locally uniformly in $q$, and the other two terms converge to $0$ uniformly in $s$, we obtain
\begin{equation*}
F_\infty(q)=|q|.
\end{equation*}
Hence \textup{(F3)} holds.

We now verify the strict radial subhomogeneity condition. Let
$$
R:=\sqrt{1+\rho^2|q|^2},
\qquad
S:=\sin(2\pi s).
$$
Then
$$
F(s,\rho q)=4+R+\frac12\frac{S}{R}.
$$
A direct computation gives
\begin{equation*}
\rho D_qF(s,\rho q)\cdot q-F(s,\rho q)
=
-4-\frac1R
-\frac{S}{2}\frac{2R^2-1}{R^3}.
\end{equation*}
Since $R\ge1$ and $|S|\le1$, we have
$
\left|
\frac{S}{2}\frac{2R^2-1}{R^3}
\right|
\le 1.
$
Therefore
\begin{equation*}
\rho D_qF(s,\rho q)\cdot q-F(s,\rho q)
\le -4+1=-3<0.
\end{equation*}
Thus \textup{(F2)} holds.

Finally,
\begin{equation*}
D_qF(s,q)
=
\frac{q}{\sqrt{1+|q|^2}}
-
\frac12\sin(2\pi s)\frac{q}{(1+|q|^2)^{3/2}}.
\end{equation*}
Hence
$
|D_qF(s,q)|\le \frac32.
$
Also
$$
|F_s(s,q)|
=
\left|
\pi\frac{\cos(2\pi s)}{\sqrt{1+|q|^2}}
\right|
\le \pi.
$$
Therefore $F$ is globally Lipschitz. Since $F$ is also smooth, \textup{(F4)} holds.
\end{example}
We finally give an explicit example with nontrivial $y$-dependence.
\begin{example}[A perturbative Hamiltonian with an explicit nonzero perturbation]
Let $F$ be the Hamiltonian in the previous example. Fix
\begin{equation*}
\eta_0=\frac{1}{10},
\qquad
\varphi(y)=-\frac{1}{4\pi^2}\sin(2\pi y_1).
\end{equation*}
Then
\begin{equation*}
\Delta_y\varphi(y)=\sin(2\pi y_1),
\qquad
D_y\varphi(y)
=
-\frac{1}{2\pi}\cos(2\pi y_1)e_1.
\end{equation*}
Define
\begin{equation*}
W_{\eta_0}(y,s,q)
=
\Delta_y\varphi(y)
+
\frac{
F(s-\eta_0\varphi(y),q-\eta_0D_y\varphi(y))-F(s,q)
}{\eta_0}.
\end{equation*}
Equivalently,
\begin{equation*}
W_{\eta_0}(y,s,q)
=
\sin(2\pi y_1)
+
10\left[
F\left(
s+\frac{\sin(2\pi y_1)}{40\pi^2},
q+\frac{\cos(2\pi y_1)}{20\pi}e_1
\right)
-
F(s,q)
\right].
\end{equation*}
Then
\begin{equation*}
H(y,s,q)=F(s,q)+\eta_0 W_{\eta_0}(y,s,q)
\end{equation*}
is a Hamiltonian belongs to the perturbative class..

The function $W_{\eta_0}$ is $1$-periodic in $y$ and $s$. Moreover, since $F$ is smooth and $\varphi$ is smooth and periodic, we have
\begin{equation*}
W_{\eta_0}\in C^\infty_{\rm loc}(\mathbb T^N\times\mathbb T\times\mathbb R^N).
\end{equation*}
In particular, $W_{\eta_0}$ satisfies the regularity assumptions required in the perturbative case.

We now verify that the explicit value
$
\eta_0=\frac{1}{10}
$
is admissible. Let $(c_0(p),\chi_p^0)$ be the unperturbed corrector associated with $F$, namely
\begin{equation*}
|p|^2(\chi_p^0)''
=
c_0(p)(\chi_p^0)'
+
F\bigl(\chi_p^0,p(\chi_p^0)'\bigr).
\end{equation*}
Define
\begin{equation*}
\chi_p^{\eta_0}(y,z)
=
\chi_p^0(z)+\eta_0\varphi(y).
\end{equation*}
Then $(\chi_p^{\eta_0})_z=(\chi_p^0)'(z)>0.$
Moreover, $D_y\chi_p^{\eta_0}
=
\eta_0D_y\varphi(y),$
and $D_y\chi_p^{\eta_0}
+
p(\chi_p^{\eta_0})_z
=
\eta_0D_y\varphi(y)+p(\chi_p^0)'(z).$
Set $s=\chi_p^{\eta_0}(y,z),
$ $
q=D_y\chi_p^{\eta_0}
+
p(\chi_p^{\eta_0})_z.$
Then $s-\eta_0\varphi(y)=\chi_p^0(z),$
and $q-\eta_0D_y\varphi(y)=p(\chi_p^0)'(z).$
By the definition of $W_{\eta_0}$,
\begin{equation*}
F(s,q)+\eta_0W_{\eta_0}(y,s,q)
=
F\bigl(s-\eta_0\varphi(y),q-\eta_0D_y\varphi(y)\bigr)
+
\eta_0\Delta_y\varphi(y).
\end{equation*}
Therefore,
\begin{equation*}
F(s,q)+\eta_0W_{\eta_0}(y,s,q)
=
F\bigl(\chi_p^0,p(\chi_p^0)'\bigr)
+
\eta_0\Delta_y\varphi(y).
\end{equation*}

On the other hand,
\begin{equation*}
\Delta_y\chi_p^{\eta_0}
=
\eta_0\Delta_y\varphi(y),
\qquad
D^2_{yz}\chi_p^{\eta_0}=0,
\qquad
(\chi_p^{\eta_0})_{zz}=(\chi_p^0)''.
\end{equation*}
Hence
\begin{equation*}
\begin{aligned}
&\Delta_y\chi_p^{\eta_0}
+
2p\cdot D^2_{yz}\chi_p^{\eta_0}
+
|p|^2(\chi_p^{\eta_0})_{zz}
\\
&\qquad
=
\eta_0\Delta_y\varphi(y)
+
|p|^2(\chi_p^0)''.
\end{aligned}
\end{equation*}
Using the unperturbed cell ODE, we obtain $c_0(p)(\chi_p^0)'
+
F\bigl(\chi_p^0,p(\chi_p^0)'\bigr)
=
|p|^2(\chi_p^0)''.$
Thus
\begin{equation*}
\begin{aligned}
&c_0(p)(\chi_p^{\eta_0})_z
+
F\left(
\chi_p^{\eta_0},
D_y\chi_p^{\eta_0}+p(\chi_p^{\eta_0})_z
\right)
\\
&\qquad
+
\eta_0W_{\eta_0}\left(
y,
\chi_p^{\eta_0},
D_y\chi_p^{\eta_0}+p(\chi_p^{\eta_0})_z
\right)
\\
&=
\Delta_y\chi_p^{\eta_0}
+
2p\cdot D^2_{yz}\chi_p^{\eta_0}
+
|p|^2(\chi_p^{\eta_0})_{zz}.
\end{aligned}
\end{equation*}
Consequently, $c_{\eta_0}(p)=c_0(p),
$ $
\chi_p^{\eta_0}(y,z)=\chi_p^0(z)+\eta_0\varphi(y)$
is an exact corrector for the perturbed cell problem. Hence the explicit value $\eta_0=\frac{1}{10}$
is admissible.
\end{example}

\appendix
\renewcommand{\theproposition}{\Alph{section}.\arabic{proposition}}
\renewcommand{\thelemma}{\Alph{section}.\arabic{lemma}}

\section{Comparison principle for equations with variable diffusion}

In the appendix, we prove the comparison principle used in the paper, following \cite{CrandallIshiiLions1992}.
\begin{proposition}[Comparison principle in the linear growth class]
Let \(T>0\). Assume that
\[
 A(x)=\sigma(x)\sigma(x)^T,
 \qquad
 \sigma\in W^{1,\infty}(\mathbb R^n;\mathbb R^{n\times m}).
\]
Assume that \(H=H(x,r,p)\) is continuous and satisfies the following Lipschitz condition: there exists \(L>0\) such that
\[
 |H(x,r,p)-H(y,s,q)|
 \le
 L\Bigl((1+|p|+|q|)|x-y|+|r-s|+|p-q|\Bigr)
\]
for all \(x,y\in\mathbb R^n\), \(r,s\in\mathbb R\), and \(p,q\in\mathbb R^n\). Let \(u\in C(\mathbb R^n\times[0,T])\) be a viscosity subsolution and \(v\in C(\mathbb R^n\times[0,T])\) be a viscosity supersolution of
\[
 w_t+H(x,w,Dw)=\operatorname{tr}(A(x)D^2w)
 \quad\text{in }\mathbb R^n\times(0,T).
\]
Assume that
\[
 u(x,0)\le v(x,0)
 \quad\text{for all }x\in\mathbb R^n,
\]
and that \(u,v\) have at most linear growth in \(x\), uniformly in \(t\in[0,T]\). Then
\[
 u(x,t)\le v(x,t)
 \quad\text{for all }(x,t)\in\mathbb R^n\times[0,T].
\]
\end{proposition}

\begin{proof}
Choose \(k>L\) and set
\[
 \widetilde u(x,t):=e^{-kt}u(x,t),
 \qquad
 \widetilde v(x,t):=e^{-kt}v(x,t).
\]
Then \(\widetilde u\) and \(\widetilde v\) are respectively a viscosity subsolution and supersolution of
\[
 w_t+G(t,x,w,Dw)=\operatorname{tr}(A(x)D^2w),
\]
where
\[
 G(t,x,r,p):=kr+e^{-kt}H(x,e^{kt}r,e^{kt}p).
\]
Moreover, if \(r\ge s\), then
\[
\begin{aligned}
 G(t,x,r,p)-G(t,x,s,p)
 &=
 k(r-s)+e^{-kt}
 \Bigl[
 H(x,e^{kt}r,e^{kt}p)-H(x,e^{kt}s,e^{kt}p)
 \Bigr] \\
 &\ge (k-L)(r-s).
\end{aligned}
\]
Set $ \lambda:=k-L>0.$

We prove \(\widetilde u\le \widetilde v\). Suppose by contradiction that $M:=\sup_{\mathbb R^n\times[0,T]}(\widetilde u-\widetilde v)>0.$

We first introduce a localization function. Since \(\widetilde u,\widetilde v\) have at most linear growth, there exists \(C_0>0\) such that
\[
 \widetilde u(x,t)-\widetilde v(y,s)\le C_0(1+|x|+|y|)
 \quad\text{for all }x,y\in\mathbb R^n,\ t,s\in[0,T].
\]
Choose a family \(\{\beta_R\}_{R>1}\subset C^2(\mathbb R^n)\) such that
\[
 \beta_R\ge0,\qquad
 \beta_R(x)=0\quad\text{for }|x|\le R,
\]
\[
 \liminf_{|x|\to\infty}\frac{\beta_R(x)}{|x|}\ge 4C_0,
\]
\[
 \|D\beta_R\|_{L^\infty}+\|D^2\beta_R\|_{L^\infty}\le C,
\]
and
\[
 \beta_R(x)\to0
 \quad\text{as }R\to\infty
 \quad\text{for every fixed }x\in\mathbb R^n.
\]
\textbf{Step 1:}

Let $ \zeta(z):=\sqrt{1+|z|^2}.$
We first prove the auxiliary estimate
\[
 \widetilde u(x,t)-\widetilde v(y,t)
 \le
 C(1+|x-y|)
 \quad\text{for all }x,y\in\mathbb R^n,\ t\in[0,T].
\]
Indeed, fix \(\Theta>0\), \(\mu>0\), and \(R>1\), and consider
\[
 \Phi_R(x,y,t)
 :=
 \widetilde u(x,t)-\widetilde v(y,t)
 -
 \Theta e^{\mu t}\zeta(x-y)
 -
 \beta_R(x)-\beta_R(y).
\]
Because of the growth of \(\beta_R\), \(\Phi_R\) attains its supremum on \(\mathbb R^n\times\mathbb R^n\times[0,T]\). We claim that for \(\Theta,\mu\) sufficiently large, independent of \(R\), $ \sup_{\mathbb R^n\times\mathbb R^n\times[0,T]}\Phi_R\le0.$
Assume otherwise. By the initial condition and by choosing \(\Theta\) larger if necessary, a positive maximum cannot occur at \(t=0\). To avoid the common time variable, for \(\nu>0\) consider
\[
 \Psi_{R,\nu}(x,y,t,s)
 :=
 \widetilde u(x,t)-\widetilde v(y,s)
 -
 \Theta e^{\mu t}\zeta(x-y)
 -
 \beta_R(x)-\beta_R(y)
 -
 \frac{|t-s|^2}{2\nu}.
\]
Let \((x_\nu,y_\nu,t_\nu,s_\nu)\) be a maximum point. As \(\nu\to0\), up to a subsequence, $(x_\nu,y_\nu,t_\nu,s_\nu)\to(\bar x,\bar y,\bar t,\bar t),$
where \((\bar x,\bar y,\bar t)\) is a positive maximum point of \(\Phi_R\).

We now derive the key estimate used in Step 1. Let
\[
 p_\nu:=\Theta e^{\mu t_\nu}D\zeta(x_\nu-y_\nu),
 \qquad
 Z_\nu:=\Theta e^{\mu t_\nu}D^2\zeta(x_\nu-y_\nu).
\]
By Ishii's lemma, there exist \(X_\nu,Y_\nu\in\mathbb S^n\) such that
\[
 \left(
 \mu\Theta e^{\mu t_\nu}\zeta(x_\nu-y_\nu)
 +
 \frac{t_\nu-s_\nu}{\nu},
 p_\nu+D\beta_R(x_\nu),
 X_\nu+D^2\beta_R(x_\nu)
 \right)
 \in
 \overline{J}^{2,+}\widetilde u(x_\nu,t_\nu),
\]
and
\[
 \left(
 \frac{t_\nu-s_\nu}{\nu},
 p_\nu-D\beta_R(y_\nu),
 Y_\nu-D^2\beta_R(y_\nu)
 \right)
 \in
 \overline{J}^{2,-}\widetilde v(y_\nu,s_\nu).
\]
Using the viscosity subsolution inequality for \(\widetilde u\), we get
\[
\begin{aligned}
 &\mu\Theta e^{\mu t_\nu}\zeta(x_\nu-y_\nu)
 +
 \frac{t_\nu-s_\nu}{\nu}
 +
 G\left(
 t_\nu,x_\nu,\widetilde u(x_\nu,t_\nu),
 p_\nu+D\beta_R(x_\nu)
 \right)
 \\
 &\le
 \operatorname{tr}
 \left(
 A(x_\nu)\bigl(X_\nu+D^2\beta_R(x_\nu)\bigr)
 \right).
\end{aligned}
\]
Similarly, using the viscosity supersolution inequality for \(\widetilde v\), we have
\[
\begin{aligned}
 &\frac{t_\nu-s_\nu}{\nu}
 +
 G\left(
 s_\nu,y_\nu,\widetilde v(y_\nu,s_\nu),
 p_\nu-D\beta_R(y_\nu)
 \right)
 \\
 &\ge
 \operatorname{tr}
 \left(
 A(y_\nu)\bigl(Y_\nu-D^2\beta_R(y_\nu)\bigr)
 \right).
\end{aligned}
\]
Subtracting the two inequalities yields
\[
\begin{aligned}
 &\mu\Theta e^{\mu t_\nu}\zeta(x_\nu-y_\nu)
 +
 G\left(
 t_\nu,x_\nu,\widetilde u(x_\nu,t_\nu),
 p_\nu+D\beta_R(x_\nu)
 \right)
 \\
 &\quad -
 G\left(
 s_\nu,y_\nu,\widetilde v(y_\nu,s_\nu),
 p_\nu-D\beta_R(y_\nu)
 \right)
 \\
 &\le
 \operatorname{tr}
 \left(
 A(x_\nu)\bigl(X_\nu+D^2\beta_R(x_\nu)\bigr)
 \right)
 -
 \operatorname{tr}
 \left(
 A(y_\nu)\bigl(Y_\nu-D^2\beta_R(y_\nu)\bigr)
 \right).
\end{aligned}
\]
We split the difference of the \(G\)-terms as
\[
\begin{aligned}
 &G\left(
 t_\nu,x_\nu,\widetilde u(x_\nu,t_\nu),
 p_\nu+D\beta_R(x_\nu)
 \right)
 -
 G\left(
 s_\nu,y_\nu,\widetilde v(y_\nu,s_\nu),
 p_\nu-D\beta_R(y_\nu)
 \right)
 \\
 &=
 \Bigl[
 G\left(
 t_\nu,x_\nu,\widetilde u(x_\nu,t_\nu),
 p_\nu+D\beta_R(x_\nu)
 \right)
 -
 G\left(
 t_\nu,x_\nu,\widetilde v(y_\nu,s_\nu),
 p_\nu+D\beta_R(x_\nu)
 \right)
 \Bigr]
 \\
 &\quad+
 \Bigl[
 G\left(
 t_\nu,x_\nu,\widetilde v(y_\nu,s_\nu),
 p_\nu+D\beta_R(x_\nu)
 \right)
 -
 G\left(
 s_\nu,y_\nu,\widetilde v(y_\nu,s_\nu),
 p_\nu-D\beta_R(y_\nu)
 \right)
 \Bigr].
\end{aligned}
\]
Since \(G\) is strictly increasing in the \(r\)-variable, we have
\[
\begin{aligned}
 &G\left(
 t_\nu,x_\nu,\widetilde u(x_\nu,t_\nu),
 p_\nu+D\beta_R(x_\nu)
 \right)
 -
 G\left(
 t_\nu,x_\nu,\widetilde v(y_\nu,s_\nu),
 p_\nu+D\beta_R(x_\nu)
 \right)
 \\
 &\ge
 \lambda\bigl(
 \widetilde u(x_\nu,t_\nu)-\widetilde v(y_\nu,s_\nu)
 \bigr),
\end{aligned}
\]
where \(\lambda=k-L>0\). Therefore
\[
\begin{aligned}
 \mu\Theta e^{\mu t_\nu}\zeta(x_\nu-y_\nu)
 &\le
 -\lambda\bigl(
 \widetilde u(x_\nu,t_\nu)-\widetilde v(y_\nu,s_\nu)
 \bigr)
 \\
 &\quad+
 E_\nu
 +
 \mathcal T_\nu,
\end{aligned}
\]
where
\[
\begin{aligned}
 E_\nu
 &:=
 \left|
 G\left(
 t_\nu,x_\nu,\widetilde v(y_\nu,s_\nu),
 p_\nu+D\beta_R(x_\nu)
 \right)
 -
 G\left(
 s_\nu,y_\nu,\widetilde v(y_\nu,s_\nu),
 p_\nu-D\beta_R(y_\nu)
 \right)
 \right|,
\end{aligned}
\]
and
\[
\begin{aligned}
 \mathcal T_\nu
 &:=
 \operatorname{tr}
 \left(
 A(x_\nu)\bigl(X_\nu+D^2\beta_R(x_\nu)\bigr)
 \right)
 -
 \operatorname{tr}
 \left(
 A(y_\nu)\bigl(Y_\nu-D^2\beta_R(y_\nu)\bigr)
 \right).
\end{aligned}
\]
We now estimate \(E_\nu\) and \(\mathcal T_\nu\). First, since $|D\zeta|\le1,
 $ $
 |p_\nu|\le \Theta e^{\mu T},$
and since $ \|D\beta_R\|_{L^\infty}\le C,$
the gradient variables in \(E_\nu\) are bounded by \(C(1+\Theta e^{\mu T})\). Moreover,
\[
 |D\beta_R(x_\nu)+D\beta_R(y_\nu)|\le C.
\]
Using the Lipschitz continuity of \(G\) in \(x\) and \(p\), we obtain
\[
 E_\nu
 \le
 C\Theta e^{\mu t_\nu}\zeta(x_\nu-y_\nu)+C+\omega_R(|t_\nu-s_\nu|),
\]
where \(\omega_R(r)\to0\) as \(r\to0\). The modulus \(\omega_R\) comes from the time-dependence of \(G\); for fixed \(R\), the variables are confined to a compact set by the localization \(\beta_R\).

Next, using \(A=\sigma\sigma^T\) and the matrix inequality in Ishii's lemma, we estimate the second-order term as follows:
\[
\begin{aligned}
 \mathcal T_\nu
 &\le
 C\Theta e^{\mu t_\nu}
 \|D^2\zeta(x_\nu-y_\nu)\|
 \|\sigma(x_\nu)-\sigma(y_\nu)\|^2
 \\
 &\quad+
 C\bigl(
 \|D^2\beta_R(x_\nu)\|
 +
 \|D^2\beta_R(y_\nu)\|
 \bigr).
\end{aligned}
\]
Since \(\sigma\) is Lipschitz and
\[
 \|D^2\zeta(z)\|\le C(1+|z|^2)^{-1/2},
\]
we get
\[
\begin{aligned}
 \mathcal T_\nu
 &\le
 C\Theta e^{\mu t_\nu}
 \frac{|x_\nu-y_\nu|^2}{\sqrt{1+|x_\nu-y_\nu|^2}}
 +
 C
 \\
 &\le
 C\Theta e^{\mu t_\nu}\zeta(x_\nu-y_\nu)+C.
\end{aligned}
\]
Combining the estimates for \(E_\nu\) and \(\mathcal T_\nu\), we arrive at
\[
\begin{aligned}
 \mu\Theta e^{\mu t_\nu}\zeta(x_\nu-y_\nu)
 &\le
 -\lambda\bigl(
 \widetilde u(x_\nu,t_\nu)-\widetilde v(y_\nu,s_\nu)
 \bigr)
 \\
 &\quad+
 C\Theta e^{\mu t_\nu}\zeta(x_\nu-y_\nu)
 +
 C
 +
 \omega_R(|t_\nu-s_\nu|).
\end{aligned}
\]
Finally, since \(t_\nu-s_\nu\to0\) as \(\nu\to0\), we may let \(\nu\to0\) and obtain
\[
\begin{aligned}
 \mu\Theta e^{\mu \bar t}\zeta(\bar x-\bar y)
 &\le
 -\lambda\bigl(
 \widetilde u(\bar x,\bar t)-\widetilde v(\bar y,\bar t)
 \bigr)
 \\
 &\quad+
 C\Theta e^{\mu \bar t}\zeta(\bar x-\bar y)
 +
 C.
\end{aligned}
\]
Moreover,
\[
 \mu\Theta e^{\mu\bar t}\zeta(\bar x-\bar y)
 \le
 C\Theta e^{\mu\bar t}\zeta(\bar x-\bar y)+C.
\]
Choosing \(\mu>2C\) and then \(\Theta\) sufficiently large yields a contradiction. Therefore $\Phi_R\le0.$
Letting \(R\to\infty\), we obtain
\[
 \widetilde u(x,t)-\widetilde v(y,t)
 \le
 \Theta e^{\mu T}\zeta(x-y)
 \le
 C(1+|x-y|).
\]

\textbf{Step 2:}

We now prove the comparison. Let \(\eta>0\), and assume that
\[
 \sup_{\mathbb R^n\times[0,T]}
 \bigl(\widetilde u-\widetilde v-\eta t\bigr)>0.
\]
For \(\alpha,\kappa,\nu>0\), consider
\[
\begin{aligned}
 M_{\alpha,\kappa,\nu}
 :=
 \sup_{x,y\in\mathbb R^n,\ t,s\in[0,T]}
 \Bigl\{
 \widetilde u(x,t)-\widetilde v(y,s)
 -
 \eta t
 -
 \alpha|x|^2
 -
 \frac{|x-y|^2}{2\kappa}
 -
 \frac{|t-s|^2}{2\nu}
 \Bigr\}.
\end{aligned}
\]
Let \((\bar x,\bar y,\bar t,\bar s)\) be a maximum point. The estimate just proved implies
\[
 |\bar x-\bar y|\le C\kappa,
\]
and the standard doubling argument gives
\[
 \frac{|\bar t-\bar s|^2}{\nu}\to0
 \quad\text{as }\nu\to0.
\]
Moreover, $ \alpha|\bar x|^2$
is bounded, hence
\[
 \alpha|\bar x|\to0
 \quad\text{as }\alpha\to0.
\]

For small parameters, the positivity of the above supremum and the initial condition imply
\[
 \bar t>0,
 \qquad
 \bar s>0.
\]
Set
\[
 p:=\frac{\bar x-\bar y}{\kappa}+2\alpha\bar x,
 \qquad
 q:=\frac{\bar x-\bar y}{\kappa}.
\]
By the parabolic Ishii lemma, there exist \(X,Y\in\mathbb S^n\) such that
\[
 \left(
 \eta+\frac{\bar t-\bar s}{\nu},
 p,
 X
 \right)
 \in
 \overline{J}^{2,+}\widetilde u(\bar x,\bar t),
\]
and
\[
 \left(
 \frac{\bar t-\bar s}{\nu},
 q,
 Y
 \right)
 \in
 \overline{J}^{2,-}\widetilde v(\bar y,\bar s).
\]
We estimate the second-order term. Since
\[
 A(x)=\sigma(x)\sigma(x)^T,
 \qquad
 \sigma\in W^{1,\infty}(\mathbb R^n;\mathbb R^{n\times m}),
\]
we write
\[
\begin{aligned}
&\operatorname{tr}(A(\bar x)X)-\operatorname{tr}(A(\bar y)Y)\\
&=
\sum_{j=1}^m
\left[
\left\langle X\sigma(\bar x)e_j,\sigma(\bar x)e_j\right\rangle
-
\left\langle Y\sigma(\bar y)e_j,\sigma(\bar y)e_j\right\rangle
\right].
\end{aligned}
\]
By the matrix inequality in Ishii's lemma,
\[
 \begin{pmatrix}
 X & 0\\
 0 & -Y
 \end{pmatrix}
 \le
 \frac{C}{\kappa}
 \begin{pmatrix}
 I & -I\\
 -I & I
 \end{pmatrix}
 +
 C\alpha
 \begin{pmatrix}
 I & 0\\
 0 & I
 \end{pmatrix}.
\]
Applying this inequality to the vector
\[
 \binom{\sigma(\bar x)e_j}{\sigma(\bar y)e_j}
\]
and summing over \(j=1,\dots,m\), we get
\[
\begin{aligned}
&\operatorname{tr}(A(\bar x)X)-\operatorname{tr}(A(\bar y)Y)\\
&\le
\frac{C}{\kappa}
\|\sigma(\bar x)-\sigma(\bar y)\|^2
+
C\alpha
\left(
\|\sigma(\bar x)\|^2+\|\sigma(\bar y)\|^2
\right).
\end{aligned}
\]
Since \(\sigma\) is bounded and Lipschitz continuous, this yields
\[
 \operatorname{tr}(A(\bar x)X)-\operatorname{tr}(A(\bar y)Y)
 \le
 C\frac{|\bar x-\bar y|^2}{\kappa}+C\alpha.
\]
Using the viscosity inequalities, we obtain
\[
\begin{aligned}
 \eta
 &+
 G(\bar t,\bar x,\widetilde u(\bar x,\bar t),p)
 -
 G(\bar s,\bar y,\widetilde v(\bar y,\bar s),q) \\
 &\le
 C\frac{|\bar x-\bar y|^2}{\kappa}+C\alpha.
\end{aligned}
\]
By the strict monotonicity of \(G\) in the \(r\)-variable,
\[
\begin{aligned}
 &G(\bar t,\bar x,\widetilde u(\bar x,\bar t),p)
 -
 G(\bar s,\bar y,\widetilde v(\bar y,\bar s),q)\\
 &\ge
 \lambda\bigl(\widetilde u(\bar x,\bar t)-\widetilde v(\bar y,\bar s)\bigr)
 -
 E_{\alpha,\kappa,\nu},
\end{aligned}
\]
where $ E_{\alpha,\kappa,\nu}
 :=
 \left|
 G(\bar t,\bar x,\widetilde v(\bar y,\bar s),p)
 -
 G(\bar s,\bar y,\widetilde v(\bar y,\bar s),q)
 \right|.$
Since
\[
 |\bar x-\bar y|\le C\kappa,
 \qquad
 |p-q|=2\alpha|\bar x|\to0,
\]
and \(p,q\) remain bounded, we have $E_{\alpha,\kappa,\nu}\to0$
as \(\nu\to0\), then \(\kappa\to0\), and finally \(\alpha\to0\). Hence
\[
 \eta
 \le
 C\frac{|\bar x-\bar y|^2}{\kappa}+C\alpha+E_{\alpha,\kappa,\nu}.
\]
Letting \(\nu\to0\), then \(\kappa\to0\), and then \(\alpha\to0\), we obtain $\eta\le0,$
which is impossible. Therefore, for every \(\eta>0\),
\[
 \widetilde u(x,t)-\widetilde v(x,t)-\eta t\le0
 \quad\text{in }\mathbb R^n\times[0,T].
\]
Letting \(\eta\to0\), we obtain
\[
 \widetilde u\le \widetilde v.
\]
Since \(e^{kt}>0\), this gives
\[
 u\le v
 \quad\text{in }\mathbb R^n\times[0,T].
\]
\end{proof}

\section{Uniform local bound for the equation}

\begin{lemma}\label{lem:initial-trace}
Assume that $H$ satisfies \textup{(H1)}--\textup{(H2)}, and let $u^\epsilon$ be the viscosity solution of
\begin{equation*}
\begin{cases}
u_t^\epsilon
+
H\left(
\dfrac{x}{\epsilon},
\dfrac{u^\epsilon}{\epsilon},
Du^\epsilon
\right)
=
\epsilon\Delta u^\epsilon
&\text{in }\mathbb R^N\times(0,\infty),
\\[2mm]
u^\epsilon(x,0)=u_0(x)
&\text{on }\mathbb R^N.
\end{cases}
\end{equation*}
Assume that
$$
u_0\in BUC(\mathbb R^N)\cap W^{1,\infty}(\mathbb R^N).
$$
Then, for every $T>0$, there exists a constant $C_T>0$, independent of $\epsilon$, such that
\begin{equation*}
|u^\epsilon(x,t)|\le C_T
\end{equation*}
for all $(x,t)\in\mathbb R^N\times[0,T]$.

Moreover, if
\begin{equation*}
u^*:={\limsup}_{\epsilon\to0}^{*}u^\epsilon,
\qquad
u_*:={\liminf}_{\epsilon\to0,*}u^\epsilon,
\end{equation*}
then
\begin{equation*}
u^*(x,0)=u_*(x,0)=u_0(x)
\end{equation*}
for every $x\in\mathbb R^N$.
\end{lemma}

\begin{proof}
Let
$
M_0:=\|u_0\|_{L^\infty(\mathbb R^N)}.
$
By \textup{(H1)}--\textup{(H2)}, we have $B_0:=
\sup_{y\in\mathbb T^N,\ s\in\mathbb T}
|H(y,s,0)|<\infty.$
Define
$$
\overline u(x,t):=M_0+B_0t,
\qquad
\underline u(x,t):=-M_0-B_0t.
$$
Then
$$
D\overline u=D\underline u=0,
\qquad
\Delta\overline u=\Delta\underline u=0.
$$
Moreover,
\begin{equation*}
\overline u_t+
H\left(
\dfrac{x}{\epsilon},
\dfrac{\overline u}{\epsilon},
D\overline u
\right)
-
\epsilon\Delta\overline u
=
B_0+
H\left(
\dfrac{x}{\epsilon},
\dfrac{\overline u}{\epsilon},
0
\right)
\ge0,
\end{equation*}
and
\begin{equation*}
\underline u_t+
H\left(
\dfrac{x}{\epsilon},
\dfrac{\underline u}{\epsilon},
D\underline u
\right)
-
\epsilon\Delta\underline u
=
-B_0+
H\left(
\dfrac{x}{\epsilon},
\dfrac{\underline u}{\epsilon},
0
\right)
\le0.
\end{equation*}
Since
$$
\underline u(x,0)\le u_0(x)\le\overline u(x,0),
$$
the comparison principle gives
\begin{equation*}
-M_0-B_0t\le u^\epsilon(x,t)\le M_0+B_0t.
\end{equation*}
This proves the uniform boundedness on every finite time interval.

We now prove the initial trace. Let
$$
L_0:=\operatorname{Lip}(u_0).
$$
For $\delta>0$, let $u_0^\delta$ be a standard smooth mollification of $u_0$. Then
\begin{equation*}
\|u_0^\delta-u_0\|_{L^\infty(\mathbb R^N)}
\le C L_0\delta,
\end{equation*}
\begin{equation*}
\|Du_0^\delta\|_{L^\infty(\mathbb R^N)}
\le L_0,
\end{equation*}
and
\begin{equation*}
\|\Delta u_0^\delta\|_{L^\infty(\mathbb R^N)}
\le
\frac{C L_0}{\delta}.
\end{equation*}
By \textup{(H1)}--\textup{(H2)}, the quantity
\begin{equation*}
B_{L_0}:=
\sup_{y\in\mathbb T^N,\ s\in\mathbb T,\ |q|\le L_0}
|H(y,s,q)|
\end{equation*}
is finite.

Define
\begin{equation*}
\overline u^{\epsilon,\delta}(x,t)
:=
u_0^\delta(x)+C L_0\delta
+
\left(
B_{L_0}
+
\epsilon\|\Delta u_0^\delta\|_{L^\infty}
\right)t.
\end{equation*}
Then
$$
D\overline u^{\epsilon,\delta}=Du_0^\delta,
\qquad
\Delta\overline u^{\epsilon,\delta}=\Delta u_0^\delta.
$$
Since
$$
|D u_0^\delta|\le L_0,
$$
we have
\begin{equation*}
\begin{aligned}
&(\overline u^{\epsilon,\delta})_t
+
H\left(
\dfrac{x}{\epsilon},
\dfrac{\overline u^{\epsilon,\delta}}{\epsilon},
D\overline u^{\epsilon,\delta}
\right)
-
\epsilon\Delta\overline u^{\epsilon,\delta}
\\
&\qquad
\ge
B_{L_0}
+
\epsilon\|\Delta u_0^\delta\|_{L^\infty}
-
B_{L_0}
-
\epsilon\|\Delta u_0^\delta\|_{L^\infty}
=
0.
\end{aligned}
\end{equation*}
Thus $\overline u^{\epsilon,\delta}$ is a supersolution. Moreover,
$$
u_0(x)\le u_0^\delta(x)+C L_0\delta
=
\overline u^{\epsilon,\delta}(x,0).
$$
By the comparison principle,
\begin{equation*}
u^\epsilon(x,t)
\le
u_0^\delta(x)+C L_0\delta
+
\left(
B_{L_0}
+
\epsilon\|\Delta u_0^\delta\|_{L^\infty}
\right)t.
\end{equation*}
Using
$$
\|u_0^\delta-u_0\|_{L^\infty}\le C L_0\delta
$$
and
$$
\|\Delta u_0^\delta\|_{L^\infty}\le \frac{C L_0}{\delta},
$$
we obtain
\begin{equation*}
u^\epsilon(x,t)
\le
u_0(x)
+
C L_0\delta
+
\left(
B_{L_0}
+
\frac{C\epsilon L_0}{\delta}
\right)t.
\end{equation*}

Similarly, define
\begin{equation*}
\underline u^{\epsilon,\delta}(x,t)
:=
u_0^\delta(x)-C L_0\delta
-
\left(
B_{L_0}
+
\epsilon\|\Delta u_0^\delta\|_{L^\infty}
\right)t.
\end{equation*}
Then $\underline u^{\epsilon,\delta}$ is a subsolution and
$$
\underline u^{\epsilon,\delta}(x,0)\le u_0(x).
$$
By the comparison principle,
\begin{equation*}
u^\epsilon(x,t)
\ge
u_0(x)
-
C L_0\delta
-
\left(
B_{L_0}
+
\frac{C\epsilon L_0}{\delta}
\right)t.
\end{equation*}
Therefore,
\begin{equation*}
|u^\epsilon(x,t)-u_0(x)|
\le
C L_0\delta
+
\left(
B_{L_0}
+
\frac{C\epsilon L_0}{\delta}
\right)t.
\end{equation*}

Let $(x_\epsilon,t_\epsilon)\to(x,0)$ as $\epsilon\to0$. Then
\begin{equation*}
|u^\epsilon(x_\epsilon,t_\epsilon)-u_0(x)|
\le
C L_0\delta
+
L_0|x_\epsilon-x|
+
\left(
B_{L_0}
+
\frac{C\epsilon L_0}{\delta}
\right)t_\epsilon.
\end{equation*}
Taking $\limsup^*$ and $\liminf_*$ as $\epsilon\to0$, then letting $\delta\to0$, we obtain
\begin{equation*}
u^*(x,0)\le u_0(x)\le u_*(x,0).
\end{equation*}
Since always
$$
u_*(x,0)\le u^*(x,0),
$$
we conclude that
\begin{equation*}
u^*(x,0)=u_*(x,0)=u_0(x).
\end{equation*}
This completes the proof.
\end{proof}

\section{Phase-shift contact lemma}

\begin{lemma}[\textbf{Phase-shift contact lemma}]
Let $Q$ be a compact cylinder and let $\phi\in C^2(Q)$. Let $\chi$ be a corrector of the form
\begin{equation*}
\chi(y,z)=z+w(y,z),
\end{equation*}
where $w$ is bounded and periodic in $(y,z)$, and
\begin{equation*}
\chi(y,z+1)=\chi(y,z)+1.
\end{equation*}
For $s\in\mathbb R$, define
\begin{equation*}
\Phi_s^\epsilon(t,x)
=
\epsilon
\chi\left(
\frac{x}{\epsilon},
\frac{\phi(t,x)}{\epsilon}+s
\right).
\end{equation*}
Then
\begin{equation*}
\Phi_{s+1}^\epsilon=\Phi_s^\epsilon+\epsilon.
\end{equation*}

Assume that $u^*-\phi$ has a strict local maximum at $(x_0,t_0)$, with $t_0>0$. Then, after replacing $Q$ by a smaller cylinder around $(x_0,t_0)$, for every sufficiently small $\epsilon$ there exists $s_\epsilon\in\mathbb R$ such that
\begin{equation*}
M_\epsilon(s_\epsilon):=
\max_{\overline Q}\left(u^\epsilon-\Phi_{s_\epsilon}^\epsilon\right)=0.
\end{equation*}
Moreover, $\epsilon s_\epsilon\to0$ 
as $\epsilon\to0$, and every maximum point $(x_\epsilon,t_\epsilon)$ of $u^\epsilon-\Phi_{s_\epsilon}^\epsilon$ in $\overline Q$ satisfies $(x_\epsilon,t_\epsilon)\to(x_0,t_0).$
In particular, for such a maximum point,
\begin{equation*}
\frac{u^\epsilon(x_\epsilon,t_\epsilon)}{\epsilon}
=
\chi\left(
\frac{x_\epsilon}{\epsilon},
\frac{\phi(x_\epsilon,t_\epsilon)}{\epsilon}
+s_\epsilon
\right).
\end{equation*}\label{phase-shift contact lemma}
\end{lemma}

\begin{proof}
Since $\chi(y,z)=z+w(y,z)$ and $w$ is bounded, we have $\Phi_s^\epsilon(t,x)
=
\phi(t,x)+\epsilon s
+
\epsilon w\left(
\frac{x}{\epsilon},
\frac{\phi(t,x)}{\epsilon}+s
\right).$
Hence
\begin{equation}
\|\Phi_s^\epsilon-\phi-\epsilon s\|_{L^\infty(Q)}
\le C\epsilon.\label{inq phase}
\end{equation}
Also, from $\chi(y,z+1)=\chi(y,z)+1$, we get $\Phi_{s+1}^\epsilon=\Phi_s^\epsilon+\epsilon.$
Therefore $M_\epsilon(s+1)=M_\epsilon(s)-\epsilon.$
Since $s\mapsto M_\epsilon(s)$ is continuous, there exists $s_\epsilon$ such that $M_\epsilon(s_\epsilon)=0.$

We now show that $\epsilon s_\epsilon\to0$. Set
$
A_\epsilon:=\max_{\overline Q}(u^\epsilon-\phi).
$
By the definition of the half-relaxed limit and by the strict maximum assumption, after subtracting a constant from $\phi$ if necessary, we may assume
$
A_\epsilon\to0.
$
Using (\ref{inq phase}),
we obtain
\begin{equation*}
A_\epsilon-\epsilon s-C\epsilon
\le
M_\epsilon(s)
\le
A_\epsilon-\epsilon s+C\epsilon.
\end{equation*}
Putting $s=s_\epsilon$ and using $M_\epsilon(s_\epsilon)=0$, we get
\begin{equation*}
|\epsilon s_\epsilon-A_\epsilon|\le C\epsilon.
\end{equation*}
Thus $\epsilon s_\epsilon\to0.$

Finally, since $u^*-\phi$ has a strict local maximum at $(x_0,t_0)$, we may choose $Q$ so small that
\begin{equation*}
\sup_{\partial Q}(u^*-\phi)
<
(u^*-\phi)(x_0,t_0).
\end{equation*}
Because $\Phi_{s_\epsilon}^\epsilon-\phi=\epsilon s_\epsilon+O(\epsilon)\to0$ uniformly on $Q$, the maximum points of
$
u^\epsilon-\Phi_{s_\epsilon}^\epsilon
$
remain in the interior of $Q$ and converge to $(x_0,t_0)$. At a maximum point $(x_\epsilon,t_\epsilon)$, the identity
$
M_\epsilon(s_\epsilon)=0
$
gives
\begin{equation*}
u^\epsilon(x_\epsilon,t_\epsilon)
=
\Phi_{s_\epsilon}^\epsilon(x_\epsilon,t_\epsilon).
\end{equation*}
Dividing by $\epsilon$ yields the desired phase identity.
\end{proof}

\section*{Acknowledgements}

I am grateful to Dr. Panrui Ni for suggesting this problem and for many helpful discussions. I sincerely thank Professor Hiroyoshi Mitake for his guidance and assistance in revising the manuscript, and Professor Hung Tran for his valuable comments.

\section*{Declarations}

\noindent {\bf Conflict of interest statement:} The author states that there is no conflict of interest.

\medskip

\noindent {\bf Data availability statement:} Data sharing not applicable to this article as no datasets were generated or analysed during the current study.

\bibliography{references}

\end{document}